\documentclass[]{gNST2e}

\newcommand{\transposee}[1]{{\vphantom{#1}}^{\mathit t}{#1}}
\newcommand{\E}{\mathbb{E}}
\newcommand{\V}{\textrm{Var}}
\newcommand{\C}{\textrm{Cov}}
\newcommand{\B}{\textrm{Bias}}
\newcommand{\R}{\mathbb{R}}
\newcommand{\N}{\mathcal{N}}

\newcommand{\D}{\displaystyle}
\def\etal{\mbox{et al.}}

\begin{document}

\title{Efficient Estimation of Sensitivity Indices}
\author{S\'ebastien Da Veiga\thanks{IFP Energies nouvelles
1 \& 4, avenue de Bois-Pr\'eau
F-92852 Rueil-Malmaison Cedex  {\tt sebastien.da-veiga@ifpen.fr}} and Fabrice Gamboa\thanks{Institut de Math\'ematiques
Universit\'e Paul Sabatier
F-31062 Toulouse Cedex 9
{\tt http://www.lsp.ups-tlse.fr/Fp/Gamboa.} {\tt gamboa@math.univ-toulouse.fr}.}}

\maketitle
\begin{abstract}
In this paper we address the problem of efficient estimation of Sobol sensitivy indices.
First, we focus on general functional integrals of conditional moments of the form $\E(\psi(\E(\varphi(Y)|X)))$ 
where $(X,Y)$ is a random vector with joint density $f$ and $\psi$ and $\varphi$ are functions that are differentiable enough.
In particular, we show that asymptotical efficient estimation of this functional boils down to the estimation of crossed quadratic functionals.  
An efficient estimate of first-order sensitivity indices is then derived as a special case.
We investigate its properties on several analytical functions and illustrate its interest on a reservoir engineering case.
\end{abstract}


\begin{classcode} 2G20, 62G06, 62G07, 62P30
\end{classcode}

\section{Introduction}

In the past decade, the increasing interest in the design and analysis of computer experiments motivated the development of dedicated and sharp statistical tools \citep{santner03}.
Design of experiments, sensitivity analysis and proxy models are examples of research fields where numerous contributions have been proposed.
More specifically, global Sensitivity Analysis (SA) is a key method for investigating complex computer codes which model physical phenomena. 
It involves a set of techniques used to quantify the influence of uncertain input parameters on the variability in numerical model responses. 
Recently, sensitivity studies have been applied in a large variety of fields, ranging from chemistry \citep{CUK73,T90} or oil recovery \citep{IMDR01} to space science \citep{Carra07} and nuclear safety \citep{IVD06}.\\
In general, global SA refers to the probabilistic framework, meaning that the uncertain input parameters are modelled as a random vector. 
By propagation, every computer code output is itself a random variable. 
Global SA techniques then consists in comparing the probability distribution of the output with the conditional probability distribution of the output when some of the inputs are fixed. 
This yields in particular useful information on the impact of some parameters. 
Such comparisons can be performed by considering various criteria, each one of them providing a different insight on the input-output relationship. 
For example, some criteria are based on distances between the probability density functions (e.g. $L^1$ and $L^2$ norms (\cite{borgo07}) or Kullback-Leibler distance (\cite{LCS06}), while others rely on functionals of conditional moments. 
Among those, variance-based methods are the most widely used \citep{salcha00}. 
They evaluate how the inputs contribute to the output variance through the so-called Sobol sensitivity indices \citep{SOB93}, which naturally emerge from a functional ANOVA decomposition of the output \citep{hoeff48,owen94,anto84}. 
Interpretation of the indices in this setting makes it possible to exhibit which input or interaction of inputs most influences the variability of the computer code output. 
This can be typically relevant for model calibration \citep{kenoha01} or model validation \citep{bayber07}.\\
Consequently, in order to conduct a sensitivity study, estimation of such sensitivity indices is of great interest. 
Initially, Monte-Carlo estimates have been proposed \citep{SOB93,MCK95}. Recent work also focused on their asymptotic properties \citep{janon12}. 
However, in many applications, calls to the computer code are very expensive, from several minutes to hours. 
In addition, the number of inputs can be large, making Monte-Carlo approaches untractable in practice. To overcome this problem, recent work focused on the use of metamodeling techniques. 
The complex computer code is approximated by a mathematical model, referred to as a "metamodel", which should be as representative as possible of the computer code, with good prediction capability. 
Once the metamodel is built and validated, it is used in the extensive Monte-Carlo sampling instead of the complex numerical model. Several metamodels can be used: polynomials, Gaussian process metamodels (\cite{OOH04}, \cite{IMDR01}) or local polynomials
(\cite{SDV09}). 
However, in these papers, the approach is generally empirical in the sense that no convergence study is performed and do not provide any insight about the asymptotic behavior of the sensitivity indices estimates. 
The only exception is the work of \citet{SDV09}, where the authors investigate the convergence of a local-polynomial based estimate using the work of \citet{FG96} and \citet{WJ94}. 
In particular, this plug-in estimate achieves a nonparametric convergence rate.\\ 
In this paper, we go one step further and propose the first asymptotically efficient estimate for sensitivity indices. 
More precisely, we investigate the problem of efficient estimation of some general nonlinear functional based on the density of a pair of random variables. 
Our approach follows the work of \citet{BL96,BL05}, and we also refer to \citet{LEV78} and \citet{KIKI96} for general results on nonlinear functionals estimation.
Such functionals of a density appear in many statistical applications and their efficient estimation remains an active research field \citep{gine2008a,gine2008b,chacon11}.
However we consider functionals involving conditional densities, which necessitate a specific treatment.
The estimate obtained here can be used for global SA involving general conditional moments, but it includes as a special case Sobol sensitivity indices.
Note also that an extension of the approach developed in our work is simultaneously proposed in the context of sliced inverse regression \citep{loubes11}.\\
The paper is organized as follows. Section \ref{sa} first recaps variance-based methods for global SA. In particular, we point out which type of nonlinear functional appears in sensitivity indices. 
Section \ref{model} then describes the theoretical framework and the proposed methodology for building an asymptotically efficient estimator. 
In Section \ref{examples}, we focus on Sobol sensitivity indices and study numerical examples showing the good behavior of the proposed estimate.
We also illustrate its interest on a reservoir engineering example, where uncertainties on the geology propagate to the potential oil recovery of a reservoir. 
Finally, all proofs are postponed to the appendix.

\section{Global sensitivity analysis} \label{sa}

In many applied fields, physicists and engineers are faced with the problem of estimating some sensitivity indices. 
These indices quantify the impact of some input variables on an output. The general situation may be formalized as follows.\\
The output $Y\in\R$ is a nonlinear regression of input variables $\boldsymbol{\tau}=(\tau_1,\ldots,\tau_l)$ ($l\geq 1$ is generally large). 
This means that $Y$ and $\boldsymbol{\tau}$ satisfy the input-output relationship 
\begin{equation}
Y=\Phi(\boldsymbol{\tau}) \label{sobol}
\end{equation}
where $\Phi$ is a known nonlinear function. Usually, $\Phi$ is complicated and has not a closed form, but it may be computed through a computer code \citep{OOH04}.  
In general, the input $\boldsymbol{\tau}$ is modelled by a random vector, so that $Y$ is also a random variable. 
A common way to quantify the impact of input variables is to use the so-called Sobol sensitivity indices \citep{SOB93}. 
Assuming that all the random variables are square integrable, the Sobol index for the input $\tau_j$ ($j=1,\ldots,l$) is
\begin{equation}
\Sigma_j=\frac{\V(\E(Y|\tau_j))}{\V(Y)}. \label{si}
\end{equation}
Observing an i.i.d. sample $(Y_1,\boldsymbol{\tau}^{(1)}),\ldots,(Y_n,\boldsymbol{\tau}^{(n)})$ (with $Y_i=\Phi(\boldsymbol{\tau}^{(i)})$, $i=1,\ldots,n$), 
the goal is is then to estimate $\Sigma_j$ ($j=1,\ldots,l$). 
Obviously, (\ref{si}) may be rewritten as
\begin{equation*}
\Sigma_j=\frac{\E(\E(Y|\tau_j)^2)-\E(Y)^2}{\V(Y)}.
\end{equation*}
Thus, in order to estimate $\Sigma_j$, the hard part is $\E(\E(Y|\tau_j)^2)$.
In this paper we will provide an asymptotically efficient estimate for this kind of quantity. 
More precisely we will tackle the problem of asymptotically efficient estimation of some general nonlinear functional.\\
Let us specify the functionals we are interested in.
Let $(Y_1,X_1),\ldots,(Y_n,X_n)$ be a sample of i.i.d. random vectors of $\R^2$ having a {\it regular} density $f$ (see Section \ref{model} for the precise frame). 
We will study the estimation of the nonlinear functional
\begin{eqnarray*}
T(f)&=& \E\Big(\psi\big(\E(\varphi(Y)|X)\big)\Big)\\
&=& \iint \psi\left(\frac{\int \varphi(y)f(x,y)dy}{\int f(x,y)dy}\right)f(x,y)dxdy
\end{eqnarray*}
where $\psi$ and $\varphi$ are regular functions. Hence, the Sobol indices are the particular case obtained with $\psi(\xi)=\xi^2$ and $\varphi(\xi)=\xi$.\\
The method developed in order to obtain an asymptotically efficient estimate for $T(f)$ follows the one developed by \citet{BL96}. 
Roughly speaking, it involves a preliminary estimate $\hat{f}$ of $f$ built on a small part of the sample. 
This preliminary estimate is used in a Taylor expansion of $T(f)$ up to the second order in a neighbourhood of $\hat{f}$. 
This expansion allows to remove the bias that occurs when using a direct plug-in method. Hence, the bias correction involves a quadratic functional of $f$. 
Due to the form of $T$, this quadratic functional of $f$ may be written as
\begin{equation*}
\theta(f)=\iiint \eta(x,y_1,y_2)f(x,y_1)f(x,y_2)dxdy_1dy_2.
\end{equation*}
This kind of functional does not fall in the frame treated in \citet{BL96} or \citet{gine2008a} and have not been studied to the best of our knowledge. 
We study this problem in Section \ref{quad} where we build an asymptotically efficient estimate for $\theta$. 
Efficient estimation of $T(f)$ is then investigated in Section \ref{main}.

\section{Model frame and method} \label{model}

Let $a<b$ and $c<d$, $L^2(dxdy)$ will denote the set of square integrable functions on $[a,b]\times [c,d]$. 
Further, $L^2(dx)$ (resp. $L^2(dy)$) will denote the set of square integrable functions on $[a,b]$ (resp. $[c,d]$). 
For sake of simplicity, we work in the whole paper with the Lebesgue measure as reference measure. 
Nevertheless, most of the results presented can be obtained for a general reference measure on $[a,b]\times [c,d]$. 
Let $(\alpha_{i_{\alpha}}(x))_{i_{\alpha}\in D_1}$ (resp. $(\beta_{i_{\beta}}(y))_{i_{\beta}\in D_2}$) be a countable orthonormal basis of $L^2(dx)$ (resp. of $L^2(dy)$). 
We set $p_i(x,y)=\alpha_{i_{\alpha}}(x) \beta_{i_{\beta}}(y)$ with $i=(i_{\alpha},i_{\beta})\in D:=D_1\times D_2$. 
Obviously $(p_i(x,y))_{i\in D}$ is a countable orthonormal (tensor) basis of $L^2(dxdy)$. We will also use the following subset of $L^2(dxdy)$ :
\begin{equation*}
\mathcal{E}=\left\{ \sum_{i\in D} e_ip_i : (e_i)_{i\in D}\ \textrm{is a sequence with} \sum_{i\in D} \left|\frac{e_i}{c_i} \right|^2 \leq 1\right\},
\end{equation*}
here $(c_i)_{i\in D}$ is a given fixed positive sequence.\\
Let $(X,Y)$ having a bounded joint density $f$ on $[a,b]\times [c,d]$ from which we have a sample $(X_i,Y_i)_{i=1,\ldots,n}$. 
We will also assume that $f$ lies in the ellipsoid $\mathcal{E}$. Recall that we wish to estimate a conditional functional
\begin{equation*}
\E\Big(\psi\big(\E(\varphi(Y)|X)\big)\Big)
\end{equation*}
where $\varphi$ is a measurable bounded function with $\chi_1\leq \varphi\leq\chi_2$ and $\psi\in C^3([\chi_1,\chi_2])$ the set of thrice continuously differentiable functions on $[\chi_1,\chi_2]$. 
This last quantity can be expressed in terms of an integral depending on the joint density $f$:
\begin{eqnarray*}
T(f) & = &\iint \psi\left(\frac{\int \varphi(y)f(x,y)dy}{\int f(x,y)dy}\right)f(x,y)dxdy.\\
& =& \iint \psi(m(x))f(x,y)dxdy
\end{eqnarray*}
where $m(x)=\int \varphi(y)f(x,y)dy/\int f(x,y)dy$ is the conditional expectation of $\varphi(Y)$ given $(X=x)$. 
We suggest as a first step to consider a preliminary estimator $\hat{f}$ of $f$, and to expand $T(f)$ in a neighborhood of $\hat{f}$. 
To achieve this goal we first define $F:[0,1]\rightarrow\R$ :
\begin{equation*}
F(u)=T(uf+(1-u)\hat{f}) \quad (u\in[0,1]).
\end{equation*}
The Taylor expansion of $F$ between $0$ and $1$ up to the third order is
\begin{equation}
F(1)=F(0)+F'(0)+\frac{1}{2}F''(0)+\frac{1}{6}F'''(\xi)(1-\xi)^3 \label{taylorF}
\end{equation}
for some $\xi\in]0,1[$. Here, we have
\begin{equation*}
F(1)=T(f)
\end{equation*}
and
\begin{eqnarray*}
F(0)=T(\hat{f})&=&\iint \psi\left(\frac{\int \varphi(y)\hat{f}(x,y)dy}{\int \hat{f}(x,y)dy}\right)\hat{f}(x,y)dxdy\\
&=&\iint \psi(\hat{m}(x))\hat{f}(x,y)dxdy
\end{eqnarray*}
where $\hat{m}(x)=\int \varphi(y)\hat{f}(x,y)dy/\int \hat{f}(x,y)dy$. Straightforward calculations also give higher-order derivatives of $F$ :
\begin{equation*}
F'(0)=\iint \left(\big[\varphi(y)-\hat{m}(x)\big]\dot\psi(\hat{m}(x))+\psi(\hat{m}(x))\right)\Big(f(x,y)-\hat{f}(x,y)\Big)dxdy\\
\end{equation*}
\begin{eqnarray*}
F''(0)&=&\iiint \frac{\ddot\psi(\hat{m}(x))}{\left(\int\hat{f}(x,y)dy\right)} \big(\hat{m}(x)-\varphi(y)\big)\big(\hat{m}(x)-\varphi(z)\big)\\
&&\Big(f(x,y)-\hat{f}(x,y)\Big)\Big(f(x,z)-\hat{f}(x,z)\Big)dxdydz\\
\end{eqnarray*}
\begin{eqnarray*}
F'''(\xi)&=&\iiiint \frac{\left(\int\hat{f}(x,y)dy\right)^2}{\left(\int
  \xi f(x,y)+(1-\xi)\hat{f}(x,y)dy\right)^{5}}\\
  &&\left[\big(\hat{m}(x)-\varphi(y)\big)\big(\hat{m}(x)-\varphi(z)\big)\big(\hat{m}(x)-\varphi(t)\big)\right.\\
  &&\left(\int\hat{f}(x,y)dy\right)\dddot\psi\left(\hat{r}(\xi,x)\right)- 3\big(\hat{m}(x)-\varphi(y)\big)\big(\hat{m}(x)-\varphi(z)\big)\\
&&\left.\left(\int [\xi f(x,y)+(1-\xi)\hat{f}(x,y)]dy\right)
    \ddot\psi\left(\hat{r}(\xi,x)\right)\right]\\
&&\Big(f(x,y)-\hat{f}(x,y)\Big)\Big(f(x,z)-\hat{f}(x,z)\Big)\\
&&\Big(f(x,t)-\hat{f}(x,t)\Big)dxdydzdt
\end{eqnarray*}
where $\D{\hat{r}(\xi,x)=\frac{\int\varphi(y)[\xi f(x,y)+(1-\xi)\hat{f}(x,y)]dy}{\int [\xi f(x,y)+(1-\xi)\hat{f}(x,y)]dy}}$ and $\dot\psi$, $\ddot\psi$ and $\dddot\psi$ denote the three first derivatives of $\psi$.\\

Plugging these expressions into (\ref{taylorF}) yields the following expansion for $T(f)$:
\vspace{0.3cm}

\fbox{
\begin{minipage}{1\textwidth} 
\begin{align}
T(f)=&\iint H(\hat{f},x,y)f(x,y)dxdy\label{taylorT}\\
&+\iiint K(\hat{f},x,y,z)f(x,y)f(x,z)dxdydz+\Gamma_n \notag 
\end{align}
\end{minipage}}

\vspace{0.3cm}

where
\begin{eqnarray*}
H(\hat{f},x,y)&=& \big[\varphi(y)-\hat{m}(x)\big]\dot\psi(\hat{m}(x))+\psi(\hat{m}(x)),\\
K(\hat{f},x,y,z)&=& \frac{1}{2}\frac{\ddot\psi(\hat{m}(x))}{\left(\int\hat{f}(x,y)dy\right)} \big(\hat{m}(x)-\varphi(y)\big)\big(\hat{m}(x)-\varphi(z)\big),\\
\Gamma_n&=&\frac{1}{6}F'''(\xi)(1-\xi)^3
\end{eqnarray*}
for some $\xi\in]0,1[$. 
Notice that the first term is a linear functional of the density $f$, it will be estimated with
\[\frac{1}{n_2}\sum_{j=1}^{n_2} H(\hat{f},X_j,Y_j).\]
The second one involves a crossed term integral which can be written as
\begin{equation}
\iiint \eta(x,y_1,y_2)f(x,y_1)f(x,y_2)dxdy_1dy_2 \label{fq}
\end{equation}
where $\eta:\R^3\rightarrow\R$ is a bounded function verifying $\eta(x,y_1,y_2)=\eta(x,y_2,y_1)$ for all $(x,y_1,y_2)\in\R^3$.
In summary, the first term can be easily estimated, unlike the second one which deserves a specific study.
In the next section we then focus on the asymptotically efficient estimation of such crossed quadratic functionals. 
In Section \ref{main}, these results are finally used to propose an asymptotically efficient estimator for $T(f)$. 

\subsection{Efficient estimation of quadratic functionals} \label{quad}

In this section, our aim is to build an asymptotically efficient estimate for 
\begin{equation*}
\theta=\iiint \eta(x,y_1,y_2)f(x,y_1)f(x,y_2)dxdy_1dy_2.
\end{equation*}
We denote $a_i=\int fp_i$ the scalar product of $f$ with $p_i$ as defined at the beginning of Section \ref{model}. We will first build a projection estimator achieving a bias equal to
\begin{equation*}
-\iiint \left[S_Mf(x,y_1)-f(x,y_1)\right]\left[S_Mf(x,y_2)-f(x,y_2)\right]\eta(x,y_1,y_2)dxdy_1dy_2
\end{equation*}
where $S_Mf=\sum_{i\in M} a_ip_i$ and $M$ is a subset of $D$. Thus, the bias would only be due to projection. Developing the previous expression leads to a goal bias equal to
\begin{eqnarray}
&&2\iiint S_Mf(x,y_1)f(x,y_2)\eta(x,y_1,y_2)dxdy_1dy_2\notag\\
&&-\iiint S_Mf(x,y_1)S_Mf(x,y_2)\eta(x,y_1,y_2)dxdy_1dy_2\notag\\
&&-\iiint f(x,y_1)f(x,y_2)\eta(x,y_1,y_2)dxdy_1dy_2. \label{biais2}
\end{eqnarray}
Consider now the estimator $\hat{\theta}_n$ defined by
\begin{eqnarray}
\hat{\theta}_{n}&=&\frac{2}{n(n-1)}\sum_{i\in M}\sum_{j\neq
  k=1}^{n}p_{i}(X_{j},Y_{j})\int p_{i}(X_{k},u)\eta(X_{k},u,Y_{k})du\notag \\
&&-\frac{1}{n(n-1)}\sum_{i,i'\in M}\sum_{j\neq
  k=1}^{n}p_{i}(X_{j},Y_{j})p_{i'}(X_{k},Y_{k})\notag \\
  &&\int
p_{i}(x,y_{1})p_{i'}(x,y_{2})\eta(x,y_{1},y_{2})dxdy_{1}dy_{2}. \label{est}
\end{eqnarray}
This estimator achieves the desired bias :
\begin{lemma}\label{biais0}
The estimator $\hat{\theta}_{n}$ defined in (\ref{est}) estimates $\theta$ with bias equal to
\begin{equation*}
-\iiint
[S_{M}f(x,y_{1})-f(x,y_{1})][S_{M}f(x,y_{2})-f(x,y_{2})]\eta(x,y_{1},y_{2})dxdy_{1}dy_{2}.
\end{equation*}
\end{lemma}

Since we will carry out an asymptotic analysis, we will work with a sequence $(M_n)_{n\geq 1}$ of subsets of $D$. We will need an extra assumption concerning this sequence: \\
\begin{itemize}
\item[A1.] For all $n\geq 1$, we can find a subset $M_n\subset D$ such that 
$\left(\sup_{i\notin M_n}|c_{i}|^{2}\right)^{2}\approx \frac{|M_n|}{n^2}$ ($A_n\approx B_n$ means $\lambda_1\leq A_n/B_n\leq \lambda_2$ for some positive constants $\lambda_1$ and $\lambda_2$). 
Furthermore, $\forall t\in L^2(dxdy)$, $\D{\int (S_{M_n}t-t)^2dxdy\rightarrow 0}$ when $n\rightarrow\infty.$
\end{itemize}

The following theorem gives the most important properties of our estimate $\hat{\theta}_n$ :
\begin{theorem}\label{tfq}
Assume A1 hold. Then $\hat{\theta}_{n}$ has the following properties:
\begin{itemize}
\item[(i)] If $|M_n|/n\rightarrow 0$ when $n\rightarrow \infty$, then
\begin{equation}
\sqrt{n}\left(\hat{\theta}_n-\theta\right)\rightarrow \N\left(0,\Lambda(f,\eta)\right), \label{na}
\end{equation}
\begin{equation}
\left| \E\left(\hat{\theta}_n-\theta\right)^2 - \Lambda(f,\eta)\right|
\leq \gamma_1\left[ \frac{|M_n|}{n}+\|S_{M_n}f-f\|_2+\|S_{M_n}g-g\|_2\right], \label{ea}
\end{equation}
where $\D{g(x,y):=\int f(x,u)\eta(x,y,u)du}$ and
$$\Lambda(f,\eta)=4  \left[ \iint g(x,y)^2f(x,y)dxdy	
-\left( \iint g(x,y)f(x,y)dxdy\right)^2\right].$$
\item[(ii)] Otherwise
\begin{equation*}
\E\left(\hat{\theta}_n-\theta\right)^2 \leq \gamma_2\frac{|M_n|}{n},
\end{equation*}
\end{itemize}
where $\gamma_1$ and $\gamma_2$ are constants depending only on $\|f\|_{\infty}$, $\|\eta\|_{\infty}$ and $\Delta_Y$ (with $\Delta_Y=d-c$). 
Moreover, these constants are increasing functions of these quantities.
\end{theorem}

\begin{remark}
Since in our main result (to be given in the next section) $\eta$ will depend on $n$ through the preliminary estimator $\hat{f}$, we need in (\ref{ea}) a bound that depends explicitly on $n$. 
Note however that (\ref{ea}) implies
\begin{equation*}
\lim_{n\rightarrow\infty} n\E\left(\hat{\theta}_n-\theta\right)^2 =\Lambda(f,\eta).
\end{equation*}
\end{remark}

 The asymptotic properties of $\hat{\theta}_n$ are of particular importance, in the sense that they are optimal as stated in the following theorem.

\begin{theorem}\label{cramerrao1}
Consider the estimation of
\begin{equation*}
\theta=\theta(f)=\iiint \eta(x,y_1,y_2)f(x,y_1)f(x,y_2)dxdy_1dy_2.
\end{equation*}
Let $f_0\in\mathcal{E}$. Then, for all estimator $\hat{\theta}_n$ of $\theta(f)$ and every family $\mathcal{V}(f_0)$ of vicinities of $f_0$, we have
\begin{equation*}
\inf_{\{\mathcal{V}(f_0)\}} \liminf_{n\rightarrow \infty} \sup_{f\in\mathcal{V}(f_0)} n\E(\hat{\theta}_n-\theta(f_0))^2\geq \Lambda(f_0,\eta).
\end{equation*}
\end{theorem}

In other words, the optimal asymptotic variance for the estimation of $\theta$ is $\Lambda(f_0,\eta)$. 
As our estimator defined in (\ref{est}) achieves this variance, it is therefore asymptotically efficient.
We are now ready to use this result to propose an efficient estimator of $T(f)$.

\subsection{Main Theorem}\label{main}

In this section we come back to our main problem of the asymptotically efficient estimation of
\begin{equation*}
T(f)=\iint \psi\left(\frac{\int \varphi(y)f(x,y)dy}{\int f(x,y)dy}\right)f(x,y)dxdy.
\end{equation*}
Recall that we have derived in (\ref{taylorT}) an expansion for $T(f)$.
The key idea is to use here the previous results on the estimation of crossed quadratic functionals. 
Indeed we have provided an asymptotically efficient estimator for the second term of this expansion, conditionally on $\hat{f}$. 
A natural and straightforward estimator for $T(f)$ is then
\begin{eqnarray*}
\widehat{T}_n&=& \frac{1}{n_2} \sum_{j=1}^{n_2} H(\hat{f},X_j,Y_j)\\
&&+ \frac{2}{n_2(n_2-1)}\sum_{i\in M}\sum_{j\neq
  k=1}^{n_2}p_{i}(X_{j},Y_{j})\int p_{i}(X_{k},u)K(\hat{f},X_{k},u,Y_{k})du\\
&&-\frac{1}{n_2(n_2-1)}\sum_{i,i'\in M}\sum_{j\neq
  k=1}^{n_2}p_{i}(X_{j},Y_{j})p_{i'}(X_{k},Y_{k})\\
  &&\int p_{i}(x,y_{1})p_{i'}(x,y_{2})K(\hat{f},x,y_{1},y_{2})dxdy_{1}dy_{2}.
\end{eqnarray*}
In the above expression, one can note that the remainder $\Gamma_n$ does not appear : we will see in the proof of the following theorem that it is negligible comparing to the two first terms.\\

 In order to study the asymptotic properties of $\widehat{T}_n$, some assumptions are required concerning the behavior of the joint density $f$ and its preliminary estimator $\hat{f}$ :
\begin{itemize}
\item[A2.] $\textrm{supp} f \subset [a,b]\times [c,d]$ and $\forall (x,y)\in \textrm{supp} f$, $0<\alpha\leq f(x,y)\leq\beta$ with $\alpha,\beta\in\R$\\
\item[A3.] One can find an estimator $\hat{f}$ of $f$ built with $n_1\approx n/\log(n)$ observations, such that 
$$\forall (x,y)\in \textrm{supp} f,\; 0<\alpha-\epsilon\leq \hat{f}(x,y)\leq\beta+\epsilon.$$
Moreover, 
$$\forall 2\leq q<+\infty,\; \forall l\in\mathbb{N}^*,\; \E_f\|\hat{f}-f\|_q^l\leq C(q,l)n_1^{-l\lambda}$$
for some $\lambda > 1/6$ and some constant $C(q,l)$ not depending on $f$ belonging to the ellipsoid $\mathcal{E}$.\\
\end{itemize}
Here $\textrm{supp} f$ denotes the set where $f$ is different from $0$. Assumption A2 is restrictive in the sense that only densities with compact support can be considered, 
excluding for example a Gaussian joint distribution.\\
Assumption A3 imposes to the estimator $\hat{f}$ a convergence fast enough towards $f$. We will use this result to control the remainder term $\Gamma_n$.\\

 We can now state the main theorem of the paper. It investigates the asymptotic properties of $\widehat{T}_n$ under assumptions A1, A2 and A3.

\begin{theorem}\label{tfec}
Assume that A1, A2 and A3 hold. Then $\widehat{T}_n$ has the following properties if $\D{\frac{|M_n|}{n}\rightarrow 0}$: 
\begin{equation}
\D{\sqrt{n}\left(\widehat{T}_n-T(f)\right)\rightarrow \N\left(0,C(f)\right)},\\ \label{na2}
\end{equation}
\begin{equation}
\lim_{n\rightarrow\infty} n\E\left(\widehat{T}_n-T(f)\right)^2 = C(f), \label{ea2}
\end{equation}
where $C(f)=\E\bigg(\V(\varphi(Y)|X)\Big[\dot\psi\big(\E(Y|X)\big)\Big]^2\bigg)+\V\Big(\psi\big(\E(\varphi(Y)|X)\big)\Big)$.\\
\end{theorem}

 We can also compute as in the previous section the semiparametric Cram\'er-Rao bound for this problem.

\begin{theorem}\label{cramerrao2}
Consider the estimation of
\begin{equation*}
T(f)=\iint\psi\left(\frac{\int \varphi(y) f(x,y)dy}{\int f(x,y)dy}\right) f(x,y)dxdy=\E\Big(\psi\big(\E(\varphi(Y)|X)\big)\Big)
\end{equation*}
for a random vector $(X,Y)$ with joint density $f\in\mathcal{E}$. Let $f_0\in\mathcal{E}$ be a density verifying the assumptions of Theorem  \ref{tfec}.
Then, for all estimator $\widehat{T}_n$ of $T(f)$ and every family $\mathcal{V}(f_0)$ of vicinities of $f_0$, we have
\begin{equation*}
\inf_{\{\mathcal{V}(f_0)\}} \liminf_{n\rightarrow \infty} \sup_{f\in\mathcal{V}(f_0)} n\E(\widehat{T}_n-T(f_0))^2\geq C(f_0).
\end{equation*}
\end{theorem}

Combination of theorems \ref{tfec} and \ref{cramerrao2} finally proves that $\widehat{T}_n$ is asymptotically efficient.

\section{Application to the estimation of sensitivity indices} \label{examples}

Now that we have built an asymptotically efficient estimate for $T(f)$, we can apply it to the particular case we were initially interested it: the estimation of Sobol sensitivity indices. 
Let us then come back to model (\ref{sobol}) :
\begin{equation*}
Y=\Phi(\boldsymbol{\tau})
\end{equation*}
where we wish to estimate (\ref{si}):
\begin{equation*}
\Sigma_j=\frac{\V(\E(Y|\tau_j))}{\V(Y)}=\frac{\E(\E(Y|\tau_j)^2)-\E(Y)^2}{\V(Y)} \quad j=1,\ldots,l.
\end{equation*}
To do so, we have an i.i.d. sample $(Y_1,\boldsymbol{\tau}^{(1)}),\ldots,(Y_n,\boldsymbol{\tau}^{(n)})$. 
We will only give here the procedure for the estimation of $\Sigma_1$ since it will be the same for the other sensitivity indices. 
Denoting $X:=\tau_1$, this problem is equivalent to estimating $\E(\E(Y|X)^2)$ with an i.i.d. sample $(Y_1,X_1),\ldots,(Y_n,X_n)$ with joint density $f$. 
We can hence apply the estimate we developed previously by letting $\psi(\xi)=\xi^2$ and $\varphi(\xi)=\xi$:
\begin{eqnarray*}
T(f)&=& \E(\E(Y|X)^2)\\
&=& \iint \left(\frac{\int yf(x,y)dy}{\int f(x,y)dy}\right)^2f(x,y)dxdy.
\end{eqnarray*}
The Taylor expansion in this case becomes
\begin{eqnarray*}
T(f)&=&\iint H(\hat{f},x,y)f(x,y)dxdy\\
&&+\iiint K(\hat{f},x,y,z)f(x,y)f(x,z)dxdydz+\Gamma_n 
\end{eqnarray*}
where
\begin{eqnarray*}
H(\hat{f},x,y)&=& 2y\hat{m}(x)-\hat{m}(x)^2,\\
K(\hat{f},x,y,z)&=&\frac{1}{\left(\int\hat{f}(x,y)dy\right)} \big(\hat{m}(x)-y\big)\big(\hat{m}(x)-z\big)
\end{eqnarray*}
and the corresponding estimator is
\begin{eqnarray*}
\widehat{T}_n&=& \frac{1}{n_2} \sum_{j=1}^{n_2} H(\hat{f},X_j,Y_j)\\
&&+ \frac{2}{n_2(n_2-1)}\sum_{i\in M}\sum_{j\neq
  k=1}^{n_2}p_{i}(X_{j},Y_{j})\int p_{i}(X_{k},u)K(\hat{f},X_{k},u,Y_{k})du\\
&&-\frac{1}{n_2(n_2-1)}\sum_{i,i'\in M}\sum_{j\neq
  k=1}^{n_2}p_{i}(X_{j},Y_{j})p_{i'}(X_{k},Y_{k})\\
  &&\int p_{i}(x,y_{1})p_{i'}(x,y_{2})K(\hat{f},x,y_{1},y_{2})dxdy_{1}dy_{2}.
\end{eqnarray*}
for some preliminary estimator $\hat{f}$ of $f$, an orthonormal basis $(p_i)_{i\in D}$ of $L^2(dxdy)$ and a subset $M\subset D$ verifying the hypotheses of Theorem \ref{tfec}.\\

We propose now to investigate the practical behavior of this estimator on two analytical models and on a reservoir engineering test case. 
In all subsequent simulation studies, the preliminary estimator $\hat{f}$ will be a kernel density estimator with bounded support built on $n_1=[\log(n)/n]$ observations. 
Moreover, we choose the Legendre polynomials on $[a,b]$ and $[c,d]$ to build the orthonormal basis $(p_i)_{i\in D}$ and we will take $|M|=\sqrt{n}$. 
Finally, the integrals in $\widehat{T}_n$ are computed with an adaptive Simpson quadrature.\\

\subsection{Simulation study on analytical functions}
The first model we investigate is
\begin{equation}
Y=\tau_1 + \tau_2^4 \label{model1}
\end{equation}
where three configurations are considered ($\tau_1$ and $\tau_2$ being independent):
\begin{itemize}
\item[(a)] $\tau_j\sim \mathcal{U}(0,1)$, $j=1,2$;
\item[(b)] $\tau_j\sim \mathcal{U}(0,3)$, $j=1,2$;
\item[(c)] $\tau_j\sim \mathcal{U}(0,5)$, $j=1,2$.
\end{itemize}
For each configuration, we report the results obtained with $n=100$ and $n=10000$ in Table \ref{tab_modele31}.
Note that we repeat the estimation 100 times with different different random samples of $\left(\tau_1,\tau_2\right)$.

\begin{table}
\tbl{Conditional moments for analytical model (\ref{model1}). Mean and standrad deviation of $\widehat{T}_n$ for different values of $n$.}
        {\begin{tabular}{@{}lccc}\toprule
            Inputs & $\E(\E(Y|\tau_j)^2)$ &  $\widehat{T}_n$ & $\widehat{T}_n$\\
            & & $n=100$ & $n=10000$\\
            \toprule
            \textbf{Configuration (a)} & & & \\
            $\tau_1$ & 0.5733 & 0.5894 +/- 0.052 & 0.5729 +/- 0.005\\
            $\tau_2$ & 0.5611 & 0.5468 +/- 0.054 & 0.5611 +/- 0.005\\
            \colrule
            \textbf{Configuration (b)} & & &\\
            $\tau_1$ & 314.04 & 305.98 +/- 52.1 & 318.27 +/- 7.52\\
            $\tau_2$ & 779.85 & 814.04 +/- 10.3 & 787.82 +/- 0.53\\
            \colrule
            \textbf{Configuration (c)} & & &\\
            $\tau_1$ & 16258 & 18414 +/- 3759 & 16897 +/- 427\\
            $\tau_2$ & 44034 & 44667 +/- 82.6 & 44073 +/- 8.17\\
            \botrule
        \end{tabular}}\label{tab_modele31}
\end{table}

The asymptotically efficient estimator $\widehat{T}_n$ gives a very accurate approximation of sensitivity indices when $n=10000$. 
But surprisingly, it also gives a reasonably accurate estimate when $n$ only equals $100$, whereas it has been built to achieve the best symptotic rate of convergence.\\

It is then interesting to compare it with other estimators, more precisely two nonparametric estimators that have been specifically built to give an accurate approximation of sensitivity indices when $n$ is not large.
The first one is based on a Gaussian process metamodel \citep{OOH04}, while the other one involves local polynomial estimators \citep{SDV09}. 
The comparison is performed on the following model :
\begin{eqnarray}
Y&=&0.2\exp(\tau_1-3)+2.2|\tau_2|+1.3\tau_2^6-2\tau_2^2-0.5\tau_2^4-0.5\tau_1^4 \notag \\
&&+2.5\tau_1^2+0.7\tau_1^3+\frac{3}{(8\tau_1-2)^2+(5\tau_2-3)^2+1}+\sin(5\tau_1)\cos(3\tau_1^2) \label{model2}
\end{eqnarray}
where $\tau_1$ and $\tau_2$ are independent and uniformly distributed on $[-1,1]$. This nonlinear function is interesting since it presents a peak and valleys. 
We estimate the sensitivity indices with a sample of size $n=100$, the results are given in Table \ref{comp1}.\\

\begin{table}
\tbl{Comparison between efficient estimation and nonparametric estimates on analytical model (\ref{model2}).}
        {\begin{tabular}{@{}lcccc}\toprule
             & True value & Oakley-O'Hagan & Local polynomials & $\widehat{T}_n$ \\
            \colrule
            $\V(\E(Y|X^1))$ & 1.0932 &  1.0539 & 1.0643 & 1.1701\\
            $\V(\E(Y|X^2))$ & 0.0729 &  0.1121 & 0.0527 & 0.0939\\
            \botrule
        \end{tabular}}\label{comp1}
\end{table}

Globally, the best estimates are given by the local polynomials technique. However, the accuracy of the asymptotically efficient estimator $\widehat{T}_n$ is comparable to that of the nonparametric ones. 
These results confirm that $\widehat{T}_n$ is a valuable estimator even with a rather complex model and a small sample size (recall that here $n=100$).\\

\subsection{Reservoir engineering example}

The PUNQ test case (Production forecasting with UNcertainty Quantification) is an oil reservoir model derived from real field data \citep{manmez01}.
The considered reservoir is surrounded by an aquifer in the north and the west, and delimited by a fault in the south and the east. 
The geological model is composed of five independent layers, three of good quality and two of poorer quality.
Six producer wells (PRO-1, PRO-4, PRO-5, PRO-11, PRO-12 and PRO-15) have been drilled, and production is supported by four additional wells injecting water (X1, X2, X3 and X4).
The geometry of the reservoir and the well locations are given in Figure \ref{punq}, left.

\begin{figure}
\begin{center}
\subfigure[]{
\resizebox*{6cm}{8cm}{\includegraphics{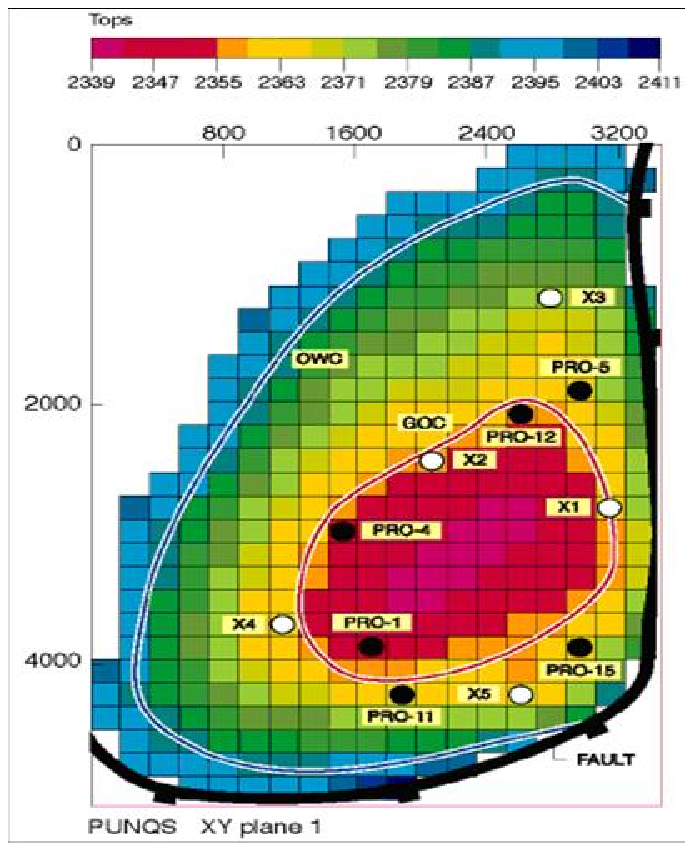}}}%
\subfigure[]{
\resizebox*{7cm}{8cm}{\includegraphics{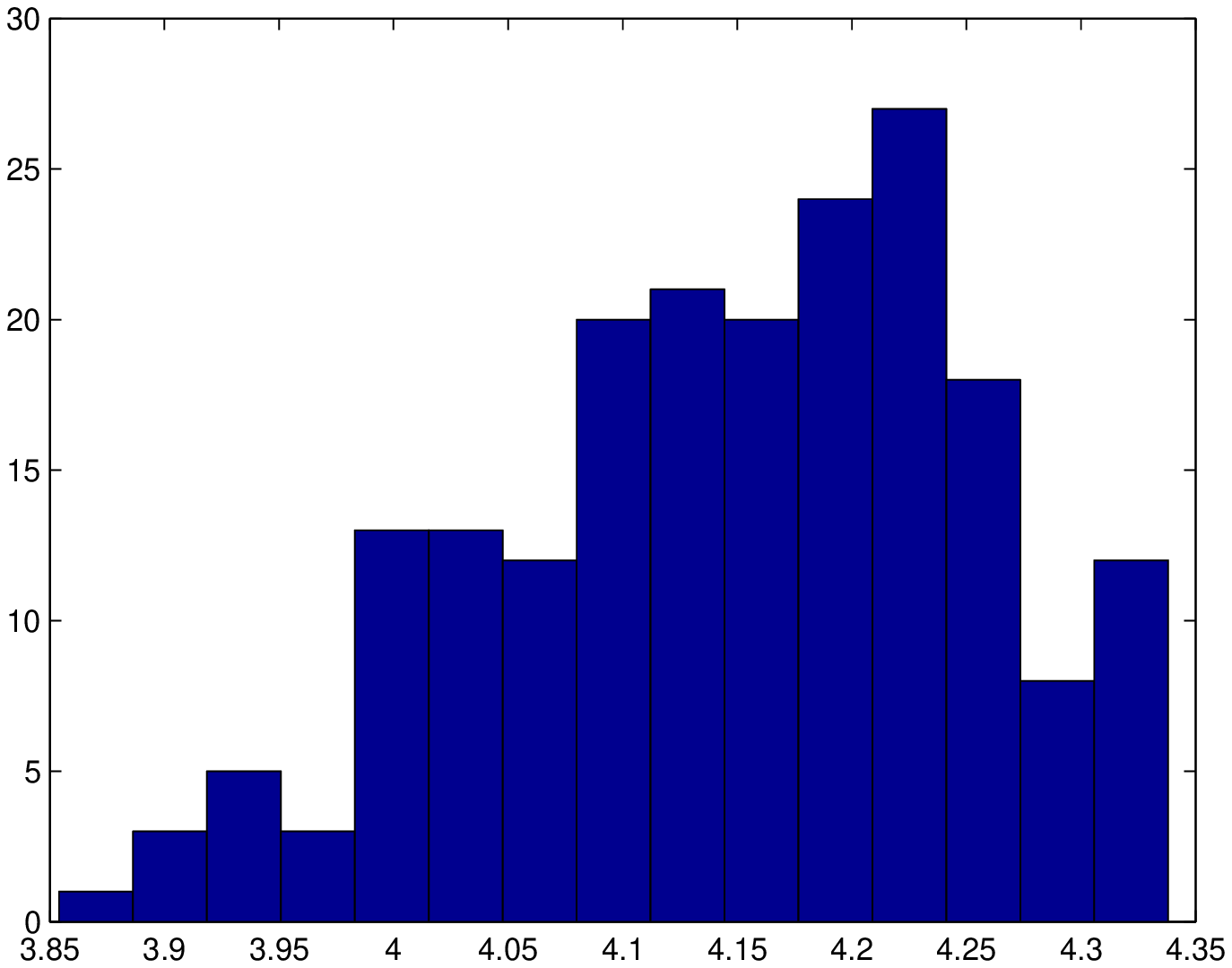}}}%
\caption{Left: top view of the PUNQ reservoir. Producer and injector wells are indicated by black and white circles, respectively.
OWC and GOW stand for Oil Water Contact and Gas Oil Contact. Right: histogram of the cumulative production after 12 years ($10^6\ m^3$).}%
\label{punq}
\end{center}
\end{figure}

In this setting, 7 variables which are characteristic of media, rocks, fluids or aquifer activity, are considered as uncertain: 
the coefficient of aquifer strength (AQUI), horizontal and vertical permeability multipliers in
good layers (MPV1 and MPH1, respectively), horizontal and vertical permeability multipliers in poor layers (MPV2 and MPH2, respectively), 
residual oil saturation after waterflood and after gas flood (SORW and SORG, respectively).
We focus here on the cumulative production of oil of this field during 12 years. 
In practice, a fluid flow simulator is used to forecast this oil production for every value of the uncertain parameters we might want to investigate.
The uncertain parameters are assumed to be uniformly distributed, with ranges given in Table \ref{tabpunq}. 
We draw a random sample of size $n = 200$ of these 7 parameters, and perform the corresponding fluid-flow simulations to compute the cumulative oil production after 12 years.
The histogram of the production obtained with this sampling is depicted in Figure \ref{punq}, right.
Clearly, the impact of the uncertain parameters on oil production is large, since different values yield forecats varying by tens of thousands of oil barrels.
In this context, reservoir engineers aim at identifying which parameters affect the most the production. 
This help them design strategies in order to reduce the most influential uncertainties, which will reduce, by propagation, the uncertainty on production forecasts.\\

In this context, computation of sensivity indices is of great interest. 
Starting from the random sample of size $n=200$, we then estimate the first-order sensitivity index of each parameter with the estimator $\widehat{T}_n$.
Results are given in Table \ref{tabpunq}. 
\begin{table}
\tbl{Range of variation and estimated first-order sensitivity index of the uncertain parameters of the PUNQ model.}
        {\begin{tabular}{@{}lcc}\toprule
            Parameter & Range of variation & Estimated sensitivity index \\
	    & & with $\widehat{T}_n$ (\%)\\
            \colrule
            AQUI & 0.2 - 0.3 & 7.206\\
	    MPH1 & 0.8 - 1.2 & 40.929\\
	    MPH2 & 0.8 - 1.2 & 0.419\\
	    MPV1 & 0.8 - 1.2 & 0.041\\
	    MPV2 & 0.8 - 1.2 & 0.693\\
	    SORG & 0.15 - 0.2 & 0.338\\
	    SORW & 0.15 - 0.25 & 49.4\\
	    \botrule
        \end{tabular}}\label{tabpunq}
\end{table}
As expected, the most influential parameters are the horizontal permeability multiplier in the good reservoir units MPH1 and the residual oil saturation after waterflood SORW.
Indeed, fluid dispacement towards the producer wells is mainly driven by the permeability in units with good petrophysical properties and by water injection.
More interestingly, vertical permeability multipliers do not seem to impact oil production in this case.
This means that fluid displacements are mainly horizontal in this reservoir.

\section{Discussion and conclusions}

In this paper, we developed a framework to build an asymptotically efficient estimate for nonlinear conditional functionals.
This estimator is both practically computable and has optimal asymptotic properties.
In particular, we show how Sobol sensitivty indices appear as a special case of our estimator.
We investigate its practical behavior on two analytical functions, and illustrate that it can compete with metamodel-based estimators.
A reservoir engineering application case is also studied, where geological and petrophysical uncertain parameters affect the forecasts on oil production.
The methodology developed here will be extended to other problems in forthcoming work. 
A very attractive extension is the construction of an adaptive procedure to calibrate the size of $M_n$ as done in \citet{BL05} for the $L^2$ norm.
However, this problem is non obvious since it would involve treating refined inequalities on U-statistics such as presented in \citet{HR02}.
From a sensitivity analysis perspective, we will also investigate efficient estimation of other indices based on entropy or other norms.
Ideally, this would give a general framework for building estimates in global sensitivity analysis.

\section*{Acknowledgements}

Many thanks are due to A. Antoniadis, B. Laurent and F. Wahl for helpful discussion.
This work has been partially supported by the French National Research Agency (ANR) through
COSINUS program (project COSTA-BRAVA ANR-09-COSI-015).

\bibliographystyle{gNST}

\appendices

\section{Proofs of Theorems}

\subsection{Proof of Lemma \ref{biais0}}

Let $\hat{\theta}_{n}=\hat{\theta}_{n}^1-\hat{\theta}_{n}^2$ where 
\begin{equation*}
\hat{\theta}_{n}^1=\frac{2}{n(n-1)}\sum_{i\in M}\sum_{j\neq
  k=1}^{n}p_{i}(X_{j},Y_{j})\int p_{i}(X_{k},u)\eta(X_{k},u,Y_{k})du
\end{equation*}
and
\begin{eqnarray*}
\hat{\theta}_{n}^2&=&\frac{1}{n(n-1)}\sum_{i,i'\in M}\sum_{j\neq
  k=1}^{n}p_{i}(X_{j},Y_{j})p_{i'}(X_{k},Y_{k})\\
  &&\int
p_{i}(x,y_{1})p_{i'}(x,y_{2})\eta(x,y_{1},y_{2})dxdy_{1}dy_{2}.
\end{eqnarray*}
Let us first compute $\E(\hat{\theta}_{n}^1)$ :
\begin{eqnarray*}
\E(\hat{\theta}_{n}^1)&=&2\sum_{i\in M} \iint p_i(x,y)f(x,y)dxdy \iiint p_i(x,y)\eta(x,u,y)f(x,y)dxdydu\\
&=& 2\sum_{i\in M} a_i \iiint p_i(x,y)\eta(x,u,y)f(x,y)dxdydu\\
&=& 2\iiint \left(\sum_{i\in M} a_i p_i(x,y)\right) \eta(x,u,y)f(x,y)dxdydu\\
&=& 2\iiint S_Mf(x,y)\eta(x,u,y)f(x,y)dxdydu.
\end{eqnarray*}
Furthermore,
\begin{eqnarray*}
\E(\hat{\theta}_{n}^2)&=&\sum_{i,i'\in M} \iint p_i(x,y)f(x,y)dxdy\iint p_{i'}(x,y)f(x,y)dxdy\\
&&\int p_{i}(x,y_{1})p_{i'}(x,y_{2})\eta(x,y_{1},y_{2})dxdy_{1}dy_{2}\\
&=&\sum_{i,i'\in M} a_ia_{i'}\int
p_{i}(x,y_{1})p_{i'}(x,y_{2})\eta(x,y_{1},y_{2})dxdy_{1}dy_{2}\\
&=& \int \left(\sum_{i\in M}a_ip_{i}(x,y_{1})\right)\left(\sum_{i'\in M}a_{i'}p_{i'}(x,y_{2})\right)\eta(x,y_{1},y_{2})dxdy_{1}dy_{2}\\
&=&\int S_Mf(x,y_1)S_Mf(x,y_2)\eta(x,y_{1},y_{2})dxdy_{1}dy_{2}.
\end{eqnarray*}
Finally, $\E(\hat{\theta}_{n})-\theta=\E(\hat{\theta}_{n}^1)-\E(\hat{\theta}_{n}^2)-\theta$ and we get the desired bias with (\ref{biais2}).

\subsection{Proof of Theorem \ref{tfq}}

We will write $M$ instead of $M_n$ for readability and denote $m=|M|$. We want to bound the precision of $\hat{\theta}_n$. We first write
$$\E\left(\hat{\theta}_{n}-\iiint
  \eta(x,y_{1},y_{2})f(x,y_{1})f(x,y_{2})dxdy_{1}dy_{2}\right)^{2}=\B^{2}(\hat{\theta}_{n})+\V(\hat{\theta}_{n}).$$
The first term of this decomposition can be easily bounded, since $\hat{\theta}_{n}$ has been built to achieve a bias equal to
\begin{eqnarray*}
\B(\hat{\theta}_{n})&=&-\iiint
[S_{M}f(x,y_{1})-f(x,y_{1})][S_{M}f(x,y_{2})-f(x,y_{2})]\\
&&\eta(x,y_{1},y_{2})dxdy_{1}dy_{2}.
\end{eqnarray*}
We then get the following lemma :
\begin{lemma}\label{biais1}
Assuming the hypotheses of Theorem \ref{tfq} hold, we have
\begin{equation*}
|\B(\hat{\theta}_{n})|\leq\Delta_{Y}\|\eta\|_{\infty}\sup_{i\notin
  M} |c_{i}|^{2}.
\end{equation*}
\end{lemma}
\begin{proof}
\begin{eqnarray*}
|\B(\hat{\theta}_{n})|&\leq&\|\eta\|_{\infty}\int \left(\int
  |S_{M}f(x,y_{1})-f(x,y_{1})|dy_{1}\right)\\
  &&\left(\int |S_{M}f(x,y_{2})-f(x,y_{2})|dy_{2}\right)dx\\
&\leq&
\|\eta\|_{\infty}\int\left(\int|S_{M}f(x,y)-f(x,y)|dy\right)^{2}dx\\
&&\leq\Delta_{Y}\|\eta\|_{\infty}\iint
(S_{M}f(x,y)-f(x,y))^{2}dxdy\\
&\leq&\Delta_{Y}\|\eta\|_{\infty}\sum_{i\notin
  M} |a_{i}|^{2}\leq\Delta_{Y}\|\eta\|_{\infty}\sup_{i\notin
  M} |c_{i}|^{2}.
\end{eqnarray*}
Indeed, $f\in \mathcal{E}$ and the last inequality follows from H\"older inequality.
\end{proof}

 Bounding the variance of $\hat{\theta}_{n}$ is however less straightforward. Let $A$ and $B$ be the $m\times 1$ vectors with components \begin{eqnarray*}
a_{i}&:=&\iint f(x,y)p_{i}(x,y)dxdy\quad i=1,\ldots,m\\
b_{i}&:=&\iiint
p_{i}(x,y_{1})f(x,y_{2})\eta(x,y_{1},y_{2})dxdy_{1}dy_{2}\\
&=&\iint g(x,y) p_{i}(x,y)dxdy\quad i=1,\ldots,m
\end{eqnarray*}
where $\D{g(x,y)=\int f(x,u)\eta(x,y,u)du}$ for each $i\in M$. $a_{i}$ et $b_{i}$ are the components of $f$ and $g$ onto the $i$th component of the basis. 
Let $Q$ and $R$ be the $m\times 1$ vectors of the centered functions $q_{i}(x,y)=p_{i}(x,y)-a_{i}$ and $\D{r_{i}(x,y)=\int p_{i}(x,u)\eta(x,u,y)du-b_{i}}$ for  $i=1,\ldots,m$. 
Let $C$ be the $m\times m$ matrix of constants $\D{c_{ii'}=\iiint
p_{i}(x,y_{1})p_{i'}(x,y_{2})\eta(x,y_{1},y_{2})dxdy_{1}dy_{2}}$ for $i,i'=1,\ldots,m$. 
Take care that here $c_{ii'}$ is double subscript unlike in the $(c_i)$ sequence appearing in the definition of the ellipsoid $\mathcal{E}$. 
We denote by $U_{n}$ the process 
$\D{U_{n}h=\frac{1}{n(n-1)}\sum_{j\neq
  k=1}^{n}h(X_{j},Y_{j},X_{k},Y_{k})}$ and by $P_{n}$ the empirical measure $\D{P_{n}f=\frac{1}{n}\sum_{j=1}^{n}f(X_{j},Y_{j})}$. 
  With the previous notation,  $\hat{\theta}_{n}$ has the following Hoeffding's decomposition (see chapter 11 of \citet{VV98}):
\begin{equation}
\hat{\theta}_{n}=U_{n}K+P_{n}L+2\transposee{A}B-\transposee{A}CA \label{thetaH}
\end{equation}
where
\begin{eqnarray*}
K(x_1,y_1,x_2,y_2)&=&2\transposee{Q}(x_1,y_1)R(x_2,y_2)-\transposee{Q}(x_1,y_1)CQ(x_2,y_2),\\
L(x_1,y_1)&=&2\transposee{A}R(x_1,y_1)+2\transposee{B}Q(x_1,y_1)-2\transposee{A}CQ(x_1,y_1).
\end{eqnarray*}
Then $\V(\hat{\theta}_{n})=\V(U_{n}K)+\V(P_{n}L)+2\;\C(U_{n}K,P_{n}L)$. We have to get bounds for each of these terms : they are given in the three following lemmas.
\begin{lemma}\label{var1}
Assuming the hypotheses of Theorem \ref{tfq} hold, we have
\begin{equation*}
\V(U_{n}K)\leq \frac{20}{n(n-1)}\|\eta\|_{\infty}^{2}\|f\|_{\infty}^{2}\Delta_{Y}^{2}(m+1).
\end{equation*}
\end{lemma}
\begin{proof}
 Since $U_{n}K$ is centered, $\V(U_{n}K)$ equals
 \begin{eqnarray*}
\E\left(\frac{1}{(n(n-1))^{2}}\sum_{j\neq k=1}^{n}\sum_{j'\neq
    k'=1}^{n}K(X_{j},Y_{j},X_{k},Y_{k})K(X_{j'},Y_{j'},X_{k'},Y_{k'})\right)\\
=\frac{1}{n(n-1)}\E(K^{2}(X_{1},Y_{1},X_{2},Y_{2})+K(X_{1},Y_{1},X_{2},Y_{2})K(X_{2},Y_{2},X_{1},Y_{1})).\\
\end{eqnarray*}
By the Cauchy-Schwarz inequality, 
$$\V(U_{n}K) \leq \frac{2}{n(n-1)}\E(K^{2}(X_{1},Y_{1},X_{2},Y_{2})).$$
Moreover, the inequality $2|\E(XY)|\leq
\E(X^{2})+\E(Y^{2})$ leads to
\begin{eqnarray*}
\E(K^{2}(X_{1},Y_{1},X_{2},Y_{2}))&\leq& 2\left[\E\left((2Q'(X_{1},Y_{1})R(X_{2},Y_{2}))^{2}\right)\right.\\
&&\left.+\E\left((Q'(X_{1},Y_{1})CQ(X_{2},Y_{2}))^{2}\right) \right].
\end{eqnarray*}
We have to bound these two terms. The first one is
$$\E\left((2Q'(X_{1},Y_{1})R(X_{2},Y_{2}))^{2}\right)=4(W_{1}-W_{2}-W_{3}+W_{4})$$
where
\begin{eqnarray*}
W_{1}&=&\int \!\!\!\int \!\!\!\int \!\!\!\int \!\!\!\int \!\!\!\int
\sum_{i,i'}p_{i}(x,y)p_{i'}(x,y)p_{i}(x',u)p_{i'}(x',v\eta(x',u,y')\eta(x',v,y')\\
&&f(x,y)f(x',y')dudvdxdydx'dy'\\
W_{2}&=&\iint\sum_{i,i'}b_{i}b_{i'}p_{i}(x,y)p_{i'}(x,y)f(x,y)dxdy\\
W_{3}&=&\iiiint
\sum_{i,i'}a_{i}a_{i'}p_{i}(x,u)p_{i'}(x,v)\eta(x,u,y)\eta(x,v,y)f(x,y)dxdy\\
W_{4}&=&\sum_{i,i'}a_{i}a_{i'}b_{i}b_{i'}.
\end{eqnarray*}
Straightforward manipulations show that $W_2\geq 0$ and $W_3\geq 0$. This implies that
$$\E\left((2Q'(X_{1},Y_{1})R(X_{2},Y_{2}))^{2}\right)\leq 4(W_{1}+W_{4}).$$
On the one hand,
\begin{eqnarray*}
W_{1}&=&\iiiint
\sum_{i,i'} p_i(x,y)p_{i'}(x,y) \int p_i(x',u)\eta(x',u,y')du\int p_{i'}(x',v)\eta(x',v,y')dvf(x,y)f(x',y')dxdydx'dy'\\
&\leq&\iiiint
\left( \sum_ip_i(x,y)\int p_i(x',u)\eta(x',u,y')du\right)^2f(x,y)f(x',y')dxdydx'dy'\\
&\leq&\|f\|_{\infty}^{2} \iiiint \left( \sum_ip_i(x,y)\int p_i(x',u)\eta(x',u,y')du\right)^2dxdydx'dy' \\
&\leq&\|f\|_{\infty}^{2}\iiiint
\sum_{i,i'} p_i(x,y)p_{i'}(x,y) \int p_i(x',u)\eta(x',u,y')du\int p_{i'}(x',v)\eta(x',v,y')dvdxdydx'dy'\\
&\leq&\|f\|_{\infty}^{2}\sum_{i,i'}  \iint
p_i(x,y)p_{i'}(x,y)dxdy\iint \left(\int p_i(x',u)\eta(x',u,y')du\right)\left(\int p_{i'}(x',v)\eta(x',v,y')dv\right) dx'dy'\\
&\leq&\|f\|_{\infty}^{2}\sum_i \iint \left(\int p_i(x',u)\eta(x',u,y')du\right)^2dx'dy'
\end{eqnarray*}
since the $p_{i}$ are orthonormal. Moreover,
\begin{eqnarray*}
\left(\int p_i(x',u)\eta(x',u,y')du\right)^2 &\leq& \left(\int p_i(x',u)^2du\right)\left(\int \eta(x',u,y')^2du\right)\\
&&\leq \|\eta\|_{\infty}^{2}\Delta_{Y}\int p_i(x',u)^2du,
\end{eqnarray*}
and then
\begin{eqnarray*}
\iint \left(\int p_i(x',u)\eta(x',u,y')du\right)^2dx'dy'&\leq &
 \|\eta\|_{\infty}^{2}\Delta_{Y}^2\iint p_i(x',u)^2dudx'\\
 && \|\eta\|_{\infty}^{2}\Delta_{Y}^2.
\end{eqnarray*}
Finally,
$$W_{1}\leq\|\eta\|_{\infty}^{2}\|f\|_{\infty}^{2}\Delta_{Y}^{2}m. $$
On the other hand,
$$W_{4}=\left(\sum_{i}a_{i}b_{i}\right)^{2}\leq \sum_{i}a_{i}^{2}\sum_{i}b_{i}^{2}\leq\|f\|_{2}^{2}\|g\|_{2}^{2}\leq\|f\|_{\infty}\|g\|_{2}^{2}.$$
By the Cauchy-Scharwz inequality we have $\|g\|_{2}^{2}\leq \|\eta\|_{\infty}^{2}\|f\|_{\infty}\Delta_{Y}^{2}$ and then
$$W_{4}\leq\|\eta\|_{\infty}^{2}\|f\|_{\infty}^{2}\Delta_{Y}^{2}$$
which leads to
$$\E\left((2Q'(X_{1},Y_{1})R(X_{2},Y_{2}))^{2}\right)\leq 4\|\eta\|_{\infty}^{2}\|f\|_{\infty}^{2}\Delta_{Y}^{2}(m+1).$$
Let us bound now the second term $\E\left((Q'(X_{1},Y_{1})CQ(X_{2},Y_{2}))^{2}\right)=W_{5}-2W_{6}+W_{7}$ where
\begin{eqnarray*}
W_{5}&=&\iiiint\sum_{i,i'}\sum_{i_{1},i'_{1}}c_{ii'}c_{i_{1}i'_{1}}p_{i}(x,y)p_{i_{1}}(x,y)p_{i'}(x',y')p_{i'_{1}}(x',y')f(x,y)f(x',y')dxdydx'dy'\\
W_{6}&=&\sum_{i,i'}\sum_{i_{1},i'_{1}}\iint
c_{ii'}c_{i_{1}i'_{1}}a_{i}a_{i_{1}}p_{i'}(x,y)p_{i'_{1}}(x,y)f(x,y)dxdy\\
W_{7}&=&\sum_{i,i'}\sum_{i_{1},i'_{1}}c_{ii'}c_{i_{1}i'_{1}}a_{i}a_{i_{1}}a_{i'}a_{i'_{1}}.
\end{eqnarray*}
Following the previous manipulations, we show that $W_6\geq 0$. Thus,
$$\E\left((Q'(X_{1},Y_{1})CQ(X_{2},Y_{2}))^{2}\right)\leq
W_{5}+W_{7}.$$First, observe that
\begin{eqnarray*}
W_{5}&=&\iiiint\left(\sum_{i,i'}c_{ii'}p_{i}(x,y)p_{i'}(x',y')\right)^{2}f(x,y)f(x',y')dxdydx'dy'\\
&\leq&\|f\|_{\infty}^{2}\iiiint\left(\sum_{i,i'}c_{ii'}p_{i}(x,y)p_{i'}(x',y')\right)^{2}dxdydx'dy'\\
&\leq&\|f\|_{\infty}^{2}\sum_{i,i'}\sum_{i_{1},i'_{1}}c_{ii'}c_{i_{1}i'_{1}}\iiiint
p_{i}(x,y)p_{i_{1}}(x,y)\\
&&p_{i'}(x',y')p_{i'_{1}}(x',y')dxdydx'dy'\\
&\leq&\|f\|_{\infty}^{2}\sum_{i,i'}c_{ii'}^{2}
\end{eqnarray*}
since the $p_{i}$ are orthonormal. Besides,
\begin{eqnarray*}
\sum_{i,i'}c_{ii'}^{2}&=&\iint \sum_{i_{\alpha},i'_{\alpha}}  \alpha_{i_{\alpha}}(x)\alpha_{i'_{\alpha}}(x) \alpha_{i_{\alpha}}(x')\alpha_{i'_{\alpha}}(x')\sum_{i_{\beta},i'_{\beta}} \left(\iint\beta_{i_{\beta}}(y_1)\beta_{i'_{\beta}}(y_2)\eta(x,y_1,y_2)dy_1dy_2\right)\\
&&\left(\iint\beta_{i_{\beta}}(y_1)\beta_{i'_{\beta}}(y_2)\eta(x',y_1,y_2)dy_1dy_2\right)dxdx'\\
&=& \iint \left( \sum_{i_{\alpha}} \alpha_{i_{\alpha}}(x)\alpha_{i_{\alpha}}(x')\right)^2\sum_{i_{\beta},i'_{\beta}} \left(\iint\beta_{i_{\beta}}(y_1)\beta_{i'_{\beta}}(y_2)\eta(x,y_1,y_2)dy_1dy_2\right)\\
&&\left(\iint\beta_{i_{\beta}}(y_1)\beta_{i'_{\beta}}(y_2)\eta(x',y_1,y_2)dy_1dy_2\right)dxdx'.
\end{eqnarray*}
But
\begin{eqnarray*}
 &&\sum_{i_{\beta},i'_{\beta}} \left(\iint\beta_{i_{\beta}}(y_1)\beta_{i'_{\beta}}(y_2)\eta(x,y_1,y_2)dy_1dy_2\right)\\
 &&\left(\iint\beta_{i_{\beta}}(y_1)\beta_{i'_{\beta}}(y_2)\eta(x',y_1,y_2)dy_1dy_2\right)\\
 &=& \sum_{i_{\beta},i'_{\beta}} \iiiint \beta_{i_{\beta}}(y_1)\beta_{i'_{\beta}}(y_2)\eta(x,y_1,y_2) \beta_{i_{\beta}}(y'_1)\beta_{i'_{\beta}}(y'_2)\eta(x',y'_1,y'_2)dy_1dy_2dy'_1dy'_2\\
 &=& \iint \sum_{i_{\beta}} \left(\int\beta_{i_{\beta}}(y_1)\eta(x,y_1,y_2)dy_1\right)\beta_{i_{\beta}}(y'_1)\sum_{i'_{\beta}}  \left(\int\beta_{i'_{\beta}}(y'_2)\eta(x',y'_1,y'_2)dy'_2\right)\beta_{i'_{\beta}}(y_2) dy'_1dy_2\\
 &=& \iint \eta(x,y'_1,y_2) \eta(x',y'_1,y_2) dy'_1dy_2\\
 &\leq& \Delta_Y^2 \|\eta\|_{\infty}^{2}
\end{eqnarray*}
using the fact that $(\beta_i)$ is an orthonormal basis. We then get
\begin{eqnarray*}
\sum_{i,i'}c_{ii'}^{2}&\leq& \Delta_Y^2 \|\eta\|_{\infty}^{2}\iint \left( \sum_{i_{\alpha}} \alpha_{i_{\alpha}}(x)\alpha_{i_{\alpha}}(x')\right)^2dxdx'\\
&\leq& \Delta_Y^2 \|\eta\|_{\infty}^{2} \iint \sum_{i_{\alpha},i'_{\alpha}} \alpha_{i_{\alpha}}(x)\alpha_{i'_{\alpha}}(x) \alpha_{i_{\alpha}}(x')\alpha_{i'_{\alpha}}(x') dxdx'\\
&\leq& \Delta_Y^2 \|\eta\|_{\infty}^{2} \sum_{i_{\alpha},i'_{\alpha}} \left(\int \alpha_{i_{\alpha}}(x)\alpha_{i'_{\alpha}}(x)dx\right)^2\\
&\leq& \Delta_Y^2 \|\eta\|_{\infty}^{2} \sum_{i_{\alpha}} \left(\int \alpha_{i_{\alpha}}(x)^2dx\right)^2\\
&\leq& \Delta_Y^2 \|\eta\|_{\infty}^{2}m
\end{eqnarray*}
since the $\alpha_{i}$ are orthonormal. Finally,
$$W_{5}\leq\|\eta\|_{\infty}^{2}\|f\|_{\infty}^{2}\Delta_{Y}^{2}m.
$$
Besides,
$$W_{7}=\left(\sum_{i,i'}c_{ii'}a_{i}a_{i'}\right)^{2}$$
with
\begin{eqnarray*}
\left|\sum_{i,i'}c_{ii'}a_{i}a_{i'}\right|&\leq&\|\eta\|_{\infty}\iiint
|S_{M}f(x,y_{1})S_{M}f(x,y_{2})|dxdy_{1}dy_{2}\\
&\leq&\|\eta\|_{\infty}\iint
\left(\int|S_{M}f(x,y_{1})S_{M}f(x,y_{2})|dx\right)dy_{1}dy_{2}.
\end{eqnarray*}
By using the Cauchy-Schwarz inequality twice, we get
\begin{eqnarray*}
\left(\sum_{i,i'}c_{ii'}a_{i}a_{i'}\right)^{2}&\leq&\Delta_{Y}^{2}\|\eta\|_{\infty}^{2}\iint
\left(\int|S_{M}f(x,y_{1})S_{M}f(x,y_{2})|dx\right)^{2}dy_{1}dy_{2}\\
&\leq&\Delta_{Y}^{2}\|\eta\|_{\infty}^{2}\iint
\left(\int S_{M}f(u,y_{1})^{2}du\right)\left(\int S_{M}f(v,y_{2})^{2}dv\right)dy_{1}dy_{2}\\
&\leq&\Delta_{Y}^{2}\|\eta\|_{\infty}^{2}\iiiint
S_{M}f(u,y_{1})^{2}S_{M}f(v,y_{2})^{2}dudvdy_{1}dy_{2}\\
&\leq&\Delta_{Y}^{2}\|\eta\|_{\infty}^{2}\left(\iint
S_{M}f(x,y)^{2}dxdy\right)^{2}\\
&\leq&\Delta_{Y}^{2}\|\eta\|_{\infty}^{2}\|f\|_{\infty}^{2}.\\
\end{eqnarray*}
Finally,
$$\E\left((Q'(X_{1},Y_{1})CQ(X_{2},Y_{2}))^{2}\right)\leq
\|\eta\|_{\infty}^{2}\|f\|_{\infty}^{2}\Delta_{Y}^{2}(m+1).$$\\
 Collecting this inequalities, we obtain
$$\V(U_{n}K)\leq \frac{20}{n(n-1)}\|\eta\|_{\infty}^{2}\|f\|_{\infty}^{2}\Delta_{Y}^{2}(m+1)$$
which concludes the proof of Lemma \ref{var1}.
\end{proof}
 Let us now deal with the second term of the Hoeffding's decomposition of $\hat{\theta}_n$ :
\begin{lemma}\label{var2}
Assuming the hypotheses of Theorem \ref{tfq} hold, we have
\begin{equation*}
\V(P_nL)\leq \frac{36}{n}\Delta_{Y}^{2}\|f\|_{\infty}^{2}\|\eta\|_{\infty}^{2}.
\end{equation*}
\end{lemma}
\begin{proof}
First note that
$$\V(P_{n}L)=\frac{1}{n}\V(L(X_{1},Y_{1})).$$
We can write $L(X_{1},Y_{1})$ as
\begin{eqnarray*}
L(X_{1},Y_{1})&=&2A'R(X_{1},Y_{1})+2B'Q(X_{1},Y_{1})-2A'CQ(X_{1},Y_{1})\\
&=& 2\sum_{i}a_{i}\left(\int
p_{i}(X_{1},u)\eta(X_{1},u,Y_{1})du-b_{i}\right)\\
&&+2\sum_{i}b_{i}(p_{i}(X_{1},Y_{1})-a_{i})-2\sum_{i,i'}c_{ii'}a_{i'}(p_{i}(X_{1},Y_{1})-a_{i})\\
&=& 2\int\sum_{i}a_{i}p_{i}(X_{1},u)\eta(X_{1},u,Y_{1})du +
2\sum_{i}b_{i}p_{i}(X_{1},Y_{1})\\
&&-2\sum_{i,i'}c_{ii'}a_{i'}p_{i}(X_{1},Y_{1})-4A'B+2A'CA\\
&=&2\int
S_{M}f(X_{1},u)\eta(X_{1},u,Y_{1})du+2S_{M}g(X_{1},Y_{1})\\
&&-2\sum_{i,i'}c_{ii'}a_{i'}p_{i}(X_{1},Y_{1})-4A'B+2A'CA.\\
\end{eqnarray*}
Let $\D{h(x,y)=\int S_{M}f(x,u)\eta(x,u,y)du}$, we have
\begin{eqnarray*} 
S_{M}h(z,t)&=&\sum_{i}\left(\iint h(x,y)p_{i}(x,y)dxdy\right)p_{i}(z,t)\\
&=&\sum_{i}\left(\iiint S_{M}f(x,u)\eta(x,u,y)p_{i}(x,y)dudxdy\right)p_{i}(z,t)\\
&=&\sum_{i,i'}\left(\iiint
  a_{i'}p_{i'}(x,u)\eta(x,u,y)p_{i}(x,y)dudxdy\right)p_{i}(z,t)\\
&=&\sum_{i,i'}c_{ii'}a_{i'}p_{i}(z,t)
\end{eqnarray*}
and we can write
$$L(X_{1},Y_{1})=2h(X_{1},Y_{1})+2S_{M}g(X_{1},Y_{1})-2S_{M}h(X_{1},Y_{1})-4A'B+2A'CA.$$
Thus,
\begin{eqnarray*}
\V(L(X_{1},Y_{1}))&=&4\V[h(X_{1},Y_{1})+S_{M}g(X_{1},Y_{1})-S_{M}h(X_{1},Y_{1})]\\
&\leq&4\E[(h(X_{1},Y_{1})+S_{M}g(X_{1},Y_{1})-S_{M}h(X_{1},Y_{1}))^{2}]\\
&\leq&12\E[(h(X_{1},Y_{1}))^{2}+(S_{M}g(X_{1},Y_{1}))^{2}+(S_{M}h(X_{1},Y_{1}))^{2}].
\end{eqnarray*}
Each of these three terms has to be bounded~:
\begin{eqnarray*}
\E((h(X_{1},Y_{1}))^{2})&=&\iint \left(\int
  S_{M}f(x,u)\eta(x,u,y)du\right)^{2}f(x,y)dxdy\\
&\leq&\Delta_{Y}\iiint S_{M}f(x,u)^{2}\eta(x,u,y)^{2}f(x,y)dxdydu\\
&\leq&\Delta_{Y}^{2}\|f\|_{\infty}\|\eta\|_{\infty}^{2}\iint
S_{M}f(x,u)^{2}dxdu\\
&\leq&\Delta_{Y}^{2}\|f\|_{\infty}\|\eta\|_{\infty}^{2}\|S_{M}f\|_{2}^{2}\\
&\leq&\Delta_{Y}^{2}\|f\|_{\infty}\|\eta\|_{\infty}^{2}\|f\|_{2}^{2}\\
&\leq&\Delta_{Y}^{2}\|f\|_{\infty}^{2}\|\eta\|_{\infty}^{2}\\
\end{eqnarray*}
\begin{equation*}
\E((S_{M}g(X_{1},Y_{1}))^{2})\leq \|f\|_{\infty}\|S_{M}g\|_{2}^{2}\leq \|f\|_{\infty}\|g\|_{2}^{2}\leq \Delta_{Y}^{2}\|f\|_{\infty}^{2}\|\eta\|_{\infty}^{2}
\end{equation*}
\begin{equation*}
\E((S_{M}h(X_{1},Y_{1}))^{2})\leq \|f\|_{\infty}\|S_{M}h\|_{2}^{2}\leq \|f\|_{\infty}\|h\|_{2}^{2}\leq\Delta_{Y}^{2}\|f\|_{\infty}^{2}\|\eta\|_{\infty}^{2}
\end{equation*}
from previous calculations. Finally,
$$\V(L(X_{1},Y_{1}))\leq 36\Delta_{Y}^{2}\|f\|_{\infty}^{2}\|\eta\|_{\infty}^{2}.$$
\end{proof}
 The last term of the Hoeffding's decomposition can also be controled :
\begin{lemma}\label{var3}
Assuming the hypotheses of Theorem \ref{tfq} hold, we have
\begin{equation*}
\C(U_{n}K,P_{n}L)=0.
\end{equation*}
\end{lemma}
\begin{proof}
Since $U_{n}K$ et $P_{n}L$ are centered, we have
\begin{eqnarray*}
\C(U_{n}K,P_{n}L)&=&\E(U_{n}KP_{n}L)\\
&=&\E\left[\frac{1}{n^{2}(n-1)}\sum_{j\neq
    k=1}^{n}K(X_{j},Y_{j},X_{k},Y_{k})\sum_{i=1}^{n}L(X_{i},Y_{i})\right]\\
&=&\frac{1}{n}\E(K(X_{1},Y_{1},X_{2},Y_{2})(L(X_{1},Y_{1})+L(X_{2},Y_{2})))\\
&=& 0
\end{eqnarray*}
since $K$, $L$, $Q$ and $R$ are centered.
\end{proof}
 The four previous lemmas give the expected result on the precision of $\hat{\theta}_n$ :
\begin{lemma}\label{precision}
Assuming the hypotheses of Theorem \ref{tfq} hold, we have :
\begin{itemize}
\item If $m/n\rightarrow 0$, 
\begin{equation*}
\E(\hat{\theta}_{n}-\theta)^{2}=O\left(\frac{1}{n}\right),
\end{equation*}
\item Otherwise,
\begin{equation*}
\E(\hat{\theta}_{n}-\theta)^{2}\leq \gamma_2(m/n^2)
\end{equation*}
where $\gamma_2$ only depends on $\|f\|_{\infty}$, $\|\eta\|_{\infty}$ and $\Delta_Y$.
\end{itemize}
\end{lemma}
\begin{proof}
Lemmas \ref{var1}, \ref{var2} and \ref{var3} imply
$$\V(\hat{\theta}_{n})\leq
\frac{20}{n(n-1)}\Delta_{Y}^{2}\|f\|_{\infty}^{2}\|\eta\|_{\infty}^{2}(m+1)+\frac{36}{n}\Delta_{Y}^{2}\|f\|_{\infty}^{2}\|\eta\|_{\infty}^{2}.$$
Finally, for $n$ large enough and a constant $\gamma\in \R$,  
$$\V(\hat{\theta}_{n})\leq \gamma
\Delta_{Y}^{2}\|f\|_{\infty}^{2}\|\eta\|_{\infty}^{2}\left(\frac{m}{n^{2}}+\frac{1}{n}\right).$$
Lemma \ref{biais1} gives 
$$\B^{2}(\hat{\theta}_{n})\leq
\Delta_{Y}^{2}\|\eta\|_{\infty}^{2}\left(\sup_{i\notin M}|c_{i}|^{2}\right)^{2}$$
and by assumption $\left(\sup_{i\notin M}|c_{i}|^{2}\right)^{2}\approx m/n^{2}$. If $m/n\rightarrow 0$, then $\E(\hat{\theta}_{n}-\theta)^{2}=O(\frac{1}{n})$. 
Otherwise $\E(\hat{\theta}_{n}-\theta)^{2}\leq \gamma_2(m/n^2)$ where $\gamma_2$ only depends on $\|f\|_{\infty}$, $\|\eta\|_{\infty}$ and $\Delta_Y$.
\end{proof}
 The lemma we just proved gives the result of Theorem \ref{tfq} when $m/n$ does not converge to $0$. 
 Let us now study more precisely the semiparametric case, that is when $\E(\hat{\theta}_{n}-\theta)^{2}=O(\frac{1}{n})$, to prove the asymptotic normality (\ref{na}) and the bound in (\ref{ea}). 
 We have $$\sqrt{n}\left(\hat{\theta}_{n}-\theta\right)=\sqrt{n}(U_{n}K)+\sqrt{n}(P_{n}L)+\sqrt{n}(2A'B-A'CA).$$ 
We will study the asymptotic behavior of each of these three terms. The first one is easily treated :
\begin{lemma}\label{var4}
Assuming the hypotheses of Theorem \ref{tfq} hold, we have
\begin{equation*}
\sqrt{n}U_nK\rightarrow 0
\end{equation*}
in probability when $n\rightarrow \infty$ if $m/n\rightarrow 0$.
\end{lemma}
\begin{proof}
Since $\D{\V(\sqrt{n}U_{n}K)\leq \frac{20}{(n-1)}\|\eta\|_{\infty}^{2}\|f\|_{\infty}^{2}\Delta_{Y}^{2}(m+1)}$, $\sqrt{n}U_nK$ converges to $0$ in probability when $n\rightarrow \infty$ if $m/n\rightarrow 0$.
\end{proof}
 The random variable $P_nL$ will be the most important term for the central limit theorem. Before studying its asymptotic normality, we need the following lemma concerning the asymptotic variance of $\sqrt{n}(P_{n}L)$ :
\begin{lemma}\label{var5}
Assuming the hypotheses of Theorem \ref{tfq} hold, we have 
\begin{equation*}
n\V(P_nL)\rightarrow \Lambda(f,\eta)
\end{equation*}
where
\begin{equation*}
\Lambda(f,\eta)=4  \left[ \iint g(x,y)^2f(x,y)dxdy	
-\left( \iint g(x,y)f(x,y)dxdy\right)^2\right].
\end{equation*}
\end{lemma}
\begin{proof}
We proved in Lemma \ref{var2} that
\begin{eqnarray*} 
\V(L(X_{1},Y_{1}))&=&4\V[h(X_{1},Y_{1})+S_{M}g(X_{1},Y_{1})-S_{M}h(X_{1},Y_{1})]\\
&=&4\V[A_1+A_2+A_3]\\
&=&4\sum_{i,j=1}^3\C(A_i,A_j).
\end{eqnarray*}
We will show that $\forall i,j\in \{1,2,3\}^2$, we have
\begin{eqnarray} 
\left| \C(A_i,A_j)-\epsilon_{ij} \left[ \iint g(x,y)^2f(x,y)dxdy	
-\left( \iint g(x,y)f(x,y)dxdy\right)^2\right]\right| \notag\\
\leq \gamma\left[ \| S_Mf-f\|_2 + \| S_Mg-g\|_2\right] \label{variances}
\end{eqnarray}
where $\epsilon_{ij}=-1$ if $i=3$ or $j=3$ and $i\neq j$ and $\epsilon_{ij}=1$ otherwise, and where $\gamma$ depends only on  $\|f\|_{\infty}$, $\|\eta\|_{\infty}$ and $\Delta_Y$.\\
We shall give the details only for the case $i=j=3$ since the calculations are similar for the other configurations. We have
$$\V(A_3)=\iint S_M^2[h(x,y)]f(x,y)dxdy-\left(\iint S_M[h(x,y)]f(x,y)dxdy\right)^2$$
We first study the quantity 
$$\left|\iint S_M^2[h(x,y)]f(x,y)dxdy-\iint g(x,y)^2f(x,y)dxdy\right|.$$
It is bounded by prout prout prout prout prout prout prout prout prout prout
\begin{eqnarray*}
&&\iint \left|S_M^2[h(x,y)]f(x,y)-S_M^2[g(x,y)]f(x,y)\right|dxdy\\
&&+\iint \left|S_M^2[g(x,y)]f(x,y)-g(x,y)^2f(x,y)\right|dxdy\\
&&\leq \|f\|_{\infty} \| S_Mh+S_Mg\|_2 \|S_Mh-S_Mg\|_2 + \|f\|_{\infty} \| S_Mg+g\|_2 \|S_Mg-g\|_2.
\end{eqnarray*}
Using the fact that $S_M$ is a projection, this sum is bounded by
\begin{eqnarray*}
~&&\|f\|_{\infty} \| h+g\|_2 \| h-g\|_2 + 2\|f\|_{\infty} \| g\|_2 \|S_Mg-g\|_2\\
~&&\leq \|f\|_{\infty} (\| h\|_2+\|g\|_2) \| h-g\|_2 + 2\|f\|_{\infty} \| g\|_2 \|S_Mg-g\|_2.
\end{eqnarray*}
We saw previously that $\|g\|_2\leq \Delta_Y \|f\|_{\infty}^{1/2} \|\eta\|_{\infty}$ and $\| h\|_2 \leq \Delta_Y \|f\|_{\infty}^{1/2} \|\eta\|_{\infty}$. The sum is then bounded by
$$2\Delta_Y\|f\|_{\infty}^{3/2} \|\eta\|_{\infty}  \| h-g\|_2 + 2\Delta_Y\|f\|_{\infty}^{3/2}\|\eta\|_{\infty}  \|S_Mg-g\|_2$$
We now have to deal with $ \| h-g\|_2$:
\begin{eqnarray*}
 \| h-g\|_2^2&=& \iint \left( \int \left(S_Mf(x,u)-f(x,u)\right)\eta(x,u,y)du\right)^2dxdy\\
 &\leq& \iint \left(\int (S_Mf(x,u)-f(x,u))^2du\right)\left(\int\eta(x,u,y)^2du\right)dxdy\\
 &\leq& \Delta_Y^2 \|\eta\|_{\infty}^{2} \|S_Mf-f\|_2^2.
\end{eqnarray*}
Finally, the sum is bounded by
$$2\Delta_Y\|f\|_{\infty}^{3/2} \|\eta\|_{\infty}  \left( \Delta_Y\|\eta\|_{\infty}\|S_Mf-f\|_2+  \|S_Mg-g\|_2\right).$$
Let us now study the second quantity
$$\left|\left(\iint S_M[h(x,y)]f(x,y)dxdy\right)^2-\left(\iint g(x,y)f(x,y)dxdy\right)^2\right|.$$
It is equal to 
\begin{eqnarray*}
\left|\left( \iint (S_M[h(x,y)]+g(x,y))f(x,y)dxdy\right)\right.\\
\left.\left( \iint (S_M[h(x,y)]-g(x,y))f(x,y)dxdy\right)\right|.
\end{eqnarray*}
By using the Cauchy-Schwarz inequality, it is bounded by
\begin{eqnarray*}
~&&\|f\|_2 \| S_Mh+g\|_2\|f\|_2 \| S_Mh-g\|_2\\
~&&\leq \|f\|_2^2 (\|h\|_2+\|g\|_2) (\| S_Mh-S_Mg\|_2+\| S_Mg-g\|_2)\\
~&&\leq 2\Delta_Y\|f\|_{\infty} ^{3/2} \|\eta\|_{\infty} (\| h-g\|_2+\| S_Mg-g\|_2)\\
~&&\leq 2\Delta_Y\|f\|_{\infty} ^{3/2} \|\eta\|_{\infty} \left(\Delta_Y\|\eta\|_{\infty}\|S_Mf-f\|_2 + \|S_Mg-g\|_2\right)
\end{eqnarray*}
by using the previous calculations. Collecting the two inequalities gives  (\ref{variances}) for $i=j=3$.\\
Finally, since by assumption $\forall t\in L^2(d\mu)$, $\|S_Mt-t\|_2 \rightarrow 0$ when $n\rightarrow\infty$, a direct consequence of (\ref{variances}) is that
\begin{eqnarray*}
&&\lim_{n\rightarrow\infty} \V(L(X_1,Y_1))\\
&& = 4  \left[ \iint g(x,y)^2f(x,y)dxdy	
-\left( \iint g(x,y)f(x,y)dxdy\right)^2\right]\\
&&= \Lambda(f,\eta).
\end{eqnarray*}
We then conclude by noting that $\V(\sqrt{n}(P_nL))=\V(L(X_1,Y_1))$.
\end{proof}
 We can now study the convergence of $\sqrt{n}(P_nL)$, which is given in the following lemma:
\begin{lemma}
Assuming the hypotheses of Theorem \ref{tfq} hold, we have
\begin{equation*}
\sqrt{n}P_nL \overset{\mathcal{L}}{\rightarrow} \N(0,\Lambda(f,\eta)).
\end{equation*}
\end{lemma}
\begin{proof}
We first note that
$$\sqrt{n}\left(P_n(2g)-2\iint g(x,y)f(x,y)dxdy\right)\rightarrow \N(0,\Lambda(f,\eta))$$
where $\D{g(x,y)=\int \eta(x,y,u)f(x,u)du}$.\\
It is then sufficient to show that the expectation of the square of
$$\D{R=\sqrt{n}\left[P_nL-\left(P_n(2g)-2 \iint g(x,y)f(x,y)dxdy\right)\right]}$$
converges to $0$. We have
\begin{eqnarray*}
\E(R^2)&=& \V(R)\\
&=& n\V(P_nL)+n\V(P_n(2g))-2n\C(P_nL,P_n(2g))
\end{eqnarray*}
We know that $n\V(P_n(2g))\rightarrow \Lambda(f,\eta)$ and Lemma  \ref{var5} shows that $n\V(P_nL)\rightarrow \Lambda(f,\eta)$. Then, we just have to prove that
$$\lim_{n\rightarrow\infty} n\C(P_nL,P_n(2g)) = \Lambda(f,\eta).$$
We have
$$n\C(P_nL,P_n(2g)) = \E(2L(X_1,Y_1)g(X_1,Y_1))$$
because $L$ is centered. Since
$$L(X_1,Y_1)=2h(X_1,Y_1)+2S_Mg(X_1,Y_1)-2S_Mh(X_1,Y_1)-4A'B+2A'CA,$$
we get
\begin{eqnarray*}
&&n\C(P_nL,P_n(2g)) = 4\iint h(x,y)g(x,y)f(x,y)dxdy\\
&& + 4\iint S_Mg(x,y)g(x,y)f(x,y)dxdy\\
&& -4\iint S_Mh(x,y)g(x,y)f(x,y)dxdy -8 \sum_i a_ib_i \iint g(x,y)f(x,y)dxdy\\
&&+ 4 A'CA \iint g(x,y)f(x,y)dxdy 
\end{eqnarray*}
which converges to $\D{4\left[ \iint g(x,y)^2f(x,y)dxdy	
-\left( \iint g(x,y)f(x,y)dudxdy\right)^2\right]}$ which is equal to $\Lambda(f,\eta)$. We finally deduce that
$$\sqrt{n}P_nL\rightarrow \N(0,\Lambda(f,\eta))$$
in distribution.
\end{proof}
 In order to prove the asymptotic normality of $\hat{\theta}_n$, the last step is to control the remainder term in the Hoeffding's decomposition:
\begin{lemma}\label{var6}
Assuming the hypotheses of Theorem \ref{tfq} hold, we have
\begin{equation*}
\sqrt{n}(2A'B-A'CA-\theta)\rightarrow 0.
\end{equation*}
\end{lemma}
\begin{proof}
$\sqrt{n}(2A'B-A'CA-\theta)\rightarrow 0$ is equal to
\begin{eqnarray*}
&&\sqrt{n}\left[2\iint g(x,y)S_Mf(x,y)dxdy\right.\\
&& - \iiint S_Mf(x,y_1)S_Mf(x,y_2)\eta(x,y_1,y_2)dxdy_1dy_2\\
&& \left.-\iiint f(x,y_1)f(x,y_2)\eta(x,y_1,y_2)dxdy_1dy_2\right].\\
\end{eqnarray*}
By replacing $g$ we get 
\begin{eqnarray*}
&&\sqrt{n}\left[2\iiint S_Mf(x,y_1)f(x,y_2)\eta(x,y_1,y_2)dxdy_1dy_2\right.\\
&& - \iiint S_Mf(x,y_1)S_Mf(x,y_2)\eta(x,y_1,y_2)dxdy_1dy_2\\
&&\left.-\iiint f(x,y_1)f(x,y_2)\eta(x,y_1,y_2)dxdy_1dy_2\right]\\
\end{eqnarray*}
With integral manipulation, we show it is also equal to
\begin{eqnarray*}
&&\sqrt{n}\left[\iiint S_Mf(x,y_1)(f(x,y_2)-S_Mf(x,y_2))\eta(x,y_1,y_2)dxdy_1dy_2\right.\\
&&\left.- \iiint f(x,y_2)(S_Mf(x,y_1)-f(x,y_1))\eta(x,y_1,y_2)dxdy_1dy_2\right]\\
&&\leq \sqrt{n} \Delta_Y \|\eta\|_{\infty}\left( \|S_Mf\|_2\|S_Mf-f\|_2+\| f\|_2\|S_Mf-f\|_2\right)\\
&&\leq 2\sqrt{n} \Delta_Y \| f\|_2\|\|\eta\|_{\infty}\|S_Mf-f\|_2\\
&&\leq 2\sqrt{n} \Delta_Y \| f\|_2\|\|\eta\|_{\infty} \left(\sup_{i\notin M}|c_{i}|^{2}\right)^{1/2}\\
&&\approx 2\Delta_Y \| f\|_2\|\|\eta\|_{\infty} \sqrt{\frac{m}{n}},
\end{eqnarray*}
which converges to $0$ when $n\rightarrow \infty$ since $m/n\rightarrow 0$.
\end{proof}
 Collecting now the results of Lemmas \ref{var4}, \ref{var5} and \ref{var6} we get (\ref{na}) since
$$\sqrt{n}\left(\hat{\theta}_{n}-\theta\right)\rightarrow \N(0,\Lambda(f,\eta))$$
in distribution. We finally have to prove (\ref{ea}). Remark that
\begin{eqnarray*}
n\E\left(\hat{\theta}_n-\theta\right)^2&=&n\B^2(\hat{\theta}_n)+n\V(\hat{\theta}_n)\\
&=& n\B^2(\hat{\theta}_n)+n\V(U_nK)+n\V(P_nL)
\end{eqnarray*}
We previously proved that
\begin{eqnarray*}
n\B^2(\hat{\theta}_n)&\leq& \lambda \Delta_Y^2\|\eta\|_{\infty}^2 \frac{m}{n} ~~\textrm{for some } \lambda\in \R,\\
n\V(U_nK)&\leq& \mu \Delta_Y^2\|f\|_{\infty}^2\|\eta\|_{\infty}^2 \frac{m}{n} ~~\textrm{for some }\mu\in \R.
\end{eqnarray*}
Moreover, (\ref{variances}) imply 
$$\left| n\V(P_nL)-\Lambda(f,\eta)\right| \leq \gamma\left[ \|S_Mf-f\|_2+\|S_Mg-g\|_2\right],$$
where $\gamma$ is a increasing function of $\|f\|_{\infty}$,$\|\eta\|_{\infty}$ and $\Delta_Y$. We then deduce (\ref{ea}) which ends the proof of Theorem \ref{tfq}.

\subsection{Proof of Theorem \ref{cramerrao1}}

To prove the inequality we will use the work of \citet{IK91} (see also chapter 25 of \citet{VV98}) on efficient estimation. 
The first step is the computation of the Fr\'echet derivative of $\theta(f)$ at a point $f_0$. 
Straightforward calculations show that
\begin{eqnarray*}
\theta(f)-\theta(f_0)&=&\iint \left[2\int \psi(x,y,z)f_0(x,z)dz\right]\left(f(x,y)-f_0(x,y)\right)dxdy\\
&& + \; O\left(\iint (f(x,y)-f_0(x,y))^2dxdy\right)
\end{eqnarray*}
from which we deduce that the Fr\'echet derivative of $\theta(f)$ at $f_0$ is\begin{equation*}
\theta'(f_0)\cdot u=\left< 2\int \psi(x,y,z)f_0(x,z)dz, u\right>\quad (u\in L^2(dxdy)),
\end{equation*}
where $\left<\cdot,\cdot\right>$ is the scalar product in $L^2(dxdy)$.
We can now use the results of \citet{IK91}. 
Denote $H(f_0)=H(f_0)=\left\{ u\in L^2(dxdy), \iint u(x,y)\sqrt{f_0(x,y)}dxdy=0\right\}$ the set of functions in $L^2(dxdy)$ orthogonal to $\sqrt{f_0}$, 
$\textrm{Proj}_{H(f_0)}$ the projection on $H(f_0)$, $A_n(t)=(\sqrt{f_0})t/\sqrt{n}$ and $P_{f_0}^{(n)}$ the joint distribution of $(X_1,\ldots,X_n)$ under $f_0$. 
Since here $X_1,\ldots,X_n$ are i.i.d., $\left\{P_f^{(n)},f\in\mathcal{E}\right\}$ is locally asymptotically normal at all points $f_0\in\mathcal{E}$ in the direction $H(f_0)$ with normalizing factor $A_n(f_0)$. 
Ibragimov and Khas'minskii result say that under these conditions, denoting $K_n=B_n\theta'(f_0)A_n\textrm{Proj}_{H(f_0)}$ with $B_n(u)=\sqrt{n}u$, if $K_n\rightarrow K$ weakly and if $K(u)=\left<t,u\right>$, 
then for every estimator $\hat{\theta}_n$ of $\theta(f)$ and every family $\mathcal{V}(f_0)$ of vicinities of $f_0$, we have
\begin{equation*}
\inf_{\{\mathcal{V}(f_0)\}} \liminf_{n\rightarrow \infty} \sup_{f\in\mathcal{V}(f_0)} n\E(\hat{\theta}_n-\theta(f_0))^2\geq \|t\|_{L^2(dxdy)}^2.
\end{equation*}
Here,
\begin{equation*}
K_n(u)=\sqrt{n}\theta'(f_0)\cdot\frac{1}{\sqrt{n}}\sqrt{f_0} \textrm{Proj}_{H(f_0)}(u)=\theta'(f_0)\cdot \left(\sqrt{f_0}\left(u-\sqrt{f_0}\int u\sqrt{f}_0\right)\right)
\end{equation*}
does not depend on $n$ and
\begin{eqnarray*}
K(u)&=& \iint \left[2\int \psi(x,y,z)f_0(x,z)dz\right] \sqrt{f_0(x,y)}\\
&& \left(u(x,y)-\sqrt{f_0(x,y)}\int u\sqrt{f}_0\right) dxdy\\
&=& \iint \left[2\int \psi(x,y,z)f_0(x,z)dz\right] \sqrt{f_0(x,y)}u(x,y)dxdy\\
&&-\iint \left[2\int \psi(x,y,z)f_0(x,z)dz\right]f_0(x,y)dxdy\int u\sqrt{f}_0\\
&=&\left<t,u\right>
\end{eqnarray*}
where
\begin{eqnarray*}
t(x,y)&=&\left[2\int \psi(x,y,z)f_0(x,z)dz\right] \sqrt{f_0(x,y)}\\
&& - \left(\iint \left[2\int \psi(x,y,z)f_0(x,z)dz\right]f_0(x,y)dxdy\right)\sqrt{f_0(x,y)}.
\end{eqnarray*}
The semiparametric Cram\'er-Rao bound for our problem is $\|t\|_{L^2(dxdy)}^2$ :
\begin{eqnarray*}
\|t\|_{L^2(dxdy)}^2&=&4 \iint \left[\int \psi(x,y,z)f_0(x,z)dz\right]^2f_0(x,y)dxdy\\
&&-4\left(\iint \left[\int \psi(x,y,z)f_0(x,z)dz\right]f_0(x,y)dxdy\right)^2\\
&=& 4 \iint g_0(x,y)^2f_0(x,y)dxdy-4\left(\iint g_0(x,y)f_0(x,y)\right)^2\\
\end{eqnarray*}
where $\D{g_0(x,y)=\int \psi(x,y,z)f_0(x,z)dz}$. Finally, we recognize the expression of $\Lambda(f_0,\psi)$ given in Theorem \ref{tfq}.

\subsection{Proof of Theorem \ref{tfec}}

We will first control the remainder term $\Gamma_n$ :
$$\Gamma_{n}=\frac{1}{6}F'''(\xi)(1-\xi)^{3}.$$
Let us recall that
\begin{eqnarray*}
F'''(\xi)&=&\iiiint \frac{\left(\int\hat{f}(x,y)dy\right)^2}{\left(\int
  \xi f(x,y)+(1-\xi)\hat{f}(x,y)dy\right)^{5}}\\
  &&\left[\big(\hat{m}(x)-\varphi(y)\big)\big(\hat{m}(x)-\varphi(z)\big)\big(\hat{m}(x)-\varphi(t)\big)\right.\\
  &&\left(\int\hat{f}(x,y)dy\right)\dddot\psi\left(\hat{r}(\xi,x)\right)- 3\big(\hat{m}(x)-\varphi(y)\big)\big(\hat{m}(x)-\varphi(z)\big)\\
&&\left.\left(\int [\xi f(x,y)+(1-\xi)\hat{f}(x,y)]dy\right)
    \ddot\psi\left(\hat{r}(\xi,x)\right)\right]\\
&&\Big(f(x,y)-\hat{f}(x,y)\Big)\Big(f(x,z)-\hat{f}(x,z)\Big)\\
&&\Big(f(x,t)-\hat{f}(x,t)\Big)dxdydzdt
\end{eqnarray*}
Assumptions A2 and A3 ensure that the first part of the integrand is bounded by a constant $\mu$ :
\begin{eqnarray*}
\Gamma_{n}&\leq&\frac{1}{6}\mu\iiiint
|f(x,y)-\hat{f}(x,y)||f(x,z)-\hat{f}(x,z)|\\
&&|f(x,t)-\hat{f}(x,t)|dxdydzdt\\
&\leq&\frac{1}{6}\mu\int
\left(\int |f(x,y)-\hat{f}(x,y)|dy\right)^{3}dx\\
&\leq& \frac{1}{6}\mu\Delta_{Y}^{2}\iint |f(x,y)-\hat{f}(x,y)|^{3}dxdy
\end{eqnarray*}
by the H\"older inequality. Then $\E(\Gamma_{n}^{2})=O(\E[(\int|f-\hat{f}|^{3})^{2}])=O(\E[\|f-\hat{f}\|_{3}^{6}])$. 
Since $\hat{f}$ verifies assumption A2, this quantity has order $O(n_{1}^{-6\lambda})$. 
If we further assume that $n_{1}\approx n/\log(n)$ and $\lambda > 1/6$, we get $E(\Gamma_{n}^{2})=o(\frac{1}{n})$, which proves that the remainder term $\Gamma_n$ is negligible. 
We will now show that $\sqrt{n}\left(\hat{T}_n-T(f)\right)$ and $Z_n=\frac{1}{n_2}\sum_{j=1}^{n_2} H(f,X_j,Y_j) - \iint H(f,x,y)f(x,y)dxdy$ have the same asymptotic behavior. 
The idea is that we can easily get a central limit theorem for $Z_n$ with asymptotic variance
\begin{equation*}
C(f)=\iint H(f,x,y)^2 f(x,y)dxdy-\left( \iint H(f,x,y) f(x,y)dxdy\right)^2,
\end{equation*}
which imply both (\ref{na2}) and (\ref{ea2}) (we will show at the end of the proof that $C(f)$ can be expressed such as in the theorem). 
In order to show that $\sqrt{n}\left(\hat{T}_n-T(f)\right)$ and $Z_n$ have the same asymptotic behavior, we will prove that
\begin{equation*}
R=\sqrt{n}\left[\hat{T}_n-T(f)-\left(\frac{1}{n_2}\sum_{j=1}^{n_2} H(f,X_j,Y_j) - \iint H(f,x,y)f(x,y)dxdy\right)\right]
\end{equation*}
has a second-order moment converging to $0$. Let us note that $R=R_1+R_2$ where
\begin{eqnarray*}
R_1&=& \sqrt{n}\left[\hat{T}_n-T(f)\right.\\
&&\left.-\left(\frac{1}{n_2}\sum_{j=1}^{n_2} H(\hat{f},X_j,Y_j) - \iint H(\hat{f},x,y)f(x,y)dxdy\right)\right],\\
R_2&=&\sqrt{n}\left[\frac{1}{n_2} \sum_{j=1}^{n_2} \left(H(\hat{f},X_j,Y_j) - \iint H(\hat{f},x,y)f(x,y)dxdy\right) \right]\\
&&- \sqrt{n}\left[\frac{1}{n_2} \sum_{j=1}^{n_2} \left(H(f,X_j,Y_j) - \iint H(f,x,y)f(x,y)dxdy\right) \right].
\end{eqnarray*}
We propose to show that both $\E(R_1^2)$ and $\E(R_2^2)$ converge to $0$.
We can write $R_1$ as follows :
\begin{equation*}
R_1= -\sqrt{n}\left[ \hat{Q}'-Q'+ \Gamma_n\right]
\end{equation*}
where
\begin{eqnarray*}
Q'&=&\iiint K(\hat{f},x,y,z)f(x,y)f(x,z),\\
K(\hat{f},x,y,z)&=& \frac{1}{2}\frac{\ddot\psi(\hat{m}(x))}{\left(\int\hat{f}(x,y)dy\right)} \big(\hat{m}(x)-\varphi(y)\big)\big(\hat{m}(x)-\varphi(z)\big)
\end{eqnarray*}
and $\hat{Q}'$ is the corresponding estimator. Since $\E\left(\Gamma_n^2\right)=o(1/n)$, we just have to control the expectation of the square of $\sqrt{n}\left[ \hat{Q}'-Q'\right]$ :
\begin{lemma}\label{qq}
Assuming the hypotheses of Theorem \ref{tfec} hold, we have
\begin{equation*}
\lim_{n\rightarrow\infty} n\E\left(\hat{Q}'-Q'\right)^2 = 0.
\end{equation*}
\end{lemma}
\begin{proof}
The bound given in (\ref{ea}) states that if $|M_n|/n\rightarrow 0$ we have
\begin{eqnarray*}
&&\left| n\E\left[\left(\hat{Q}'-Q'\right)^2|\hat{f}\right]\right.\\
&&\left. - 4  \left[ \iint \hat{g}(x,y)^2f(x,y)dxdy	
-\left( \iint \hat{g}(x,y)f(x,y)dxdy\right)^2\right]\right|\\
&&\leq \gamma_1(\|f\|_{\infty},\|\psi\|_{\infty},\Delta_Y) \left[ \frac{|M_n|}{n}+\|S_{M}f-f\|_2+\|S_{M}\hat{g}-\hat{g}\|_2\right]
\end{eqnarray*}
where $\D{\hat{g}(x,y)=\int K(\hat{f},x,y,z)f(x,z)dz}$. By deconditioning, we get
\begin{eqnarray*}
&&\left| n\E\left[\left(\hat{Q}'-Q'\right)^2\right]\right.\\
&&\left. - 4  \E\left[ \iint \hat{g}(x,y)^2f(x,y)dxdy	
-\left( \iint \hat{g}(x,y)f(x,y)dxdy\right)^2\right]\right|\\
&&\leq \gamma_1(\|f\|_{\infty},\|\psi\|_{\infty},\Delta_Y) \left[ \frac{|M_n|}{n}+\|S_{M}f-f\|_2+\E\left(\|S_{M}\hat{g}-\hat{g}\|_2\right)\right].
\end{eqnarray*}
Note that
\begin{eqnarray*}
\E\left(\|S_{M}\hat{g}-\hat{g}\|_2\right)&\leq&\E\left(\|S_{M}\hat{g}-S_Mg\|_2\right) + \E\left(\|S_{M}g-g\|_2\right)\\
&\leq& \E\left(\|\hat{g}-g\|_2\right) + \E\left(\|S_{M}g-g\|_2\right)
\end{eqnarray*}
where $\D{g(x,y)=\int K(f,x,y,z)f(x,z)dz}$. The second term converges to $0$ since $g\in L^2(dxdy)$ and $\forall t\in L^2(dxdy)$, $\int (S_{M}t-t)^2d\mu\rightarrow 0$. Moreover
\begin{eqnarray*}
\|\hat{g}-g\|_2^2&=&  \iint \left[\hat{g}(x,y)-g(x,y)\right]^2f(x,y)dxdy\\
&=&  \iint \left[ \int \left(K(\hat{f},x,y,z)-K(f,x,y,z)\right)f(x,z)dz\right]^2f(x,y)dxdy\\
&\leq&  \iint \left[ \int \left(K(\hat{f},x,y,z)-K(f,x,y,z)\right)^2dz\right]\\
&&\left[\int f(x,z)^2dz\right]f(x,y)dxdy\\
&\leq& \Delta_Y^2\|f\|_{\infty}^3 \iiint \left(K(\hat{f},x,y,z)-K(f,x,y,z)\right)^2dxdz\\
&\leq& \delta \Delta_Y^3\|f\|_{\infty}^3 \iint (f(x,y)-\hat{f}(x,y))^2dxdy
\end{eqnarray*}
for some constant $\delta$ by applying the mean value theorem to $K(f,x,y,z)-K(\hat{f},x,y,z)$. Of course, the bound $\delta$ is obtained here by considering assumptions A1, A2 and A3. 
Since $\E(\|f-\hat{f}\|_2)\rightarrow 0$, we get $\E\left(\|\hat{g}-g\|_2\right)\rightarrow 0$. Let us now show that the expectation of
$$\iint \hat{g}(x,y)^2f(x,y)dxdy-\left( \iint \hat{g}(x,y)f(x,y)dxdy\right)^2$$
converges to 0. We will only develop the proof for the first term :
\begin{eqnarray*}
&&\left|\iint \hat{g}(x,y)^2f(x,y)dxdy- \iint g(x,y)^2f(x,y)dxdy\right|\\
&&\leq \iint \left|\hat{g}(x,y)^2-g(x,y)^2\right|f(x,y)dxdy\\
&&\leq \lambda \iint \left(\hat{g}(x,y)-g(x,y)\right)^2dxdy\\
&&\leq \lambda \|\hat{g}-g\|_2^2
\end{eqnarray*}
for some constant $\lambda$. By taking the expectation of both sides, we see it is enough to show that $\E\left(\|\hat{g}-g\|_2^2\right)\rightarrow 0$, which is done exactly as above. Besides, we can verify that
\begin{eqnarray*}
g(x,y)&=& \int K(f,x,y,z)f(x,z)dz\\
&=& \frac{1}{2}\frac{\ddot\psi(m(x))}{\left(\int f(x,y)dy\right)} \big(m(x)-\varphi(y)\big)\\
&&\left(m(x)\int f(x,z)dz-\int \varphi(z)f(x,z)dz\right)\\
&=& 0,
\end{eqnarray*}
which proves that the expectation of$\D{\iint \hat{g}(x,y)^2f(x,y)dxdy}$ converges to $0$. 
Similar considerations show that the expectation of the second term $\D{\left( \iint \hat{g}(x,y)f(x,y)dxdy\right)^2}$ also converges to  $0$. 
We finally have
$$\lim_{n\rightarrow\infty} n\E\left(\hat{Q}'-Q'\right)^2 = 0.$$
\end{proof}
 Lemma \ref{qq} imply that $\E(R_1^2)\rightarrow 0$. We will now prove that $\E(R_2^2)\rightarrow 0$ :
\begin{eqnarray*}
\E(R_2^2)&=&\frac{n}{n_2} \E\left[ \iint \left(H(f,x,y)-H(\hat{f},x,y)\right)^2 f(x,y)dxdy\right]\\
&&- \frac{n}{n_2} \E\left[ \iint H(f,x,y)f(x,y)dxdy - \iint H(\hat{f},x,y)f(x,y)dxdy\right]^2.
\end{eqnarray*}
The same arguments as before (mean value theorem and assumptions A2 and A3) show that $\E(R_2^2)\rightarrow 0$.
At last, we can give another expression for the asymptotic variance :
\begin{equation*}
C(f)=\iint H(f,x,y)^2 f(x,y)dxdy-\left( \iint H(f,x,y) f(x,y)dxdy\right)^2.
\end{equation*}
We will prove that 
$$C(f)=\E\left(\V(\varphi(Y)|X)\left[\dot\psi\left(\E(Y|X)\right)\right]^2\right)+\V\left(\psi\left(\E(\varphi(Y)|X)\right)\right).$$ 
Remark that 
\begin{eqnarray}
\iint H(f,x,y) f(x,y)dxdy&=& \iint \left( \left[ \varphi(y)-m(x)\right] \dot\psi(m(x)) + \psi(m(x))\right)f(x,y)dxdy\notag \\
&=& \iint m(x) \dot\psi(m(x)) f(x,y)dxdy- \iint m(x) \dot\psi(m(x))f(x,y)dxdy\notag\\
&&+ \iint \psi(m(x)) f(x,y)dxdy\notag \\
&=& \E\left(\psi\left(\E(\varphi(Y)|X)\right)\right). \label{hec}
\end{eqnarray}
Moreover,
\begin{eqnarray*}
H(f,x,y)^2&=&\left[ \varphi(y)-m(x)\right]^2 \dot\psi(m(x))^2 + \psi(m(x))^2+2\left[ \varphi(y)-m(x)\right] \dot\psi(m(x))\psi(m(x))\\
&=& \varphi(y)^2\dot\psi(m(x))^2 + m(x)^2\dot\psi(m(x))^2 -2\varphi(y)m(x)\dot\psi(m(x))^2\\
&& + \psi(m(x))^2+2\left[ \varphi(y)-m(x)\right] \dot\psi(m(x))\psi(m(x)).
\end{eqnarray*}
We can then rewrite $\D{\iint H(f,x,y)^2f(x,y)dxdy}$ as:
\begin{eqnarray*}
&& \iint  \varphi(y)^2\dot\psi(m(x))^2f(x,y)dxdy+ \iint m(x)^2\dot\psi(m(x))^2f(x,y)dxdy\\
&&-2\iint \varphi(y)m(x)\dot\psi(m(x))^2f(x,y)dxdy + \iint \psi(m(x))^2f(x,y)dxdy\\
&&+2\iint \varphi(y)\dot\psi(m(x))\psi(m(x))f(x,y)dxdy-2\iint m(x)\dot\psi(m(x))\psi(m(x))f(x,y)dxdy\\
&=& \iint v(x)\dot\psi(m(x))^2f(x,y)dxdy - \iint m(x)^2\dot\psi(m(x))f(x,y)dxdy + \iint \psi(m(x))^2f(x,y)dxdy\\
&=& \iint \left(\left[v(x) -m(x)^2\right]\dot\psi(m(x))^2 + \psi(m(x))^2\right)f(x,y)dxdy\\
&=& \E\left(\left[v(X) -m(X)^2\right]\dot\psi(m(X))^2\right) + \E\left(\psi(m(X))^2\right)\\
&=& \E\left(\left[\E(\varphi(Y)^2|X) -\E(\varphi(Y)|X)^2\right]\left[\dot\psi(\E(\varphi(Y)|X))\right]^2\right)+ \E\left(\psi(\E(\varphi(Y)|X))^2\right)\\
&=& \E\left(\V(\varphi(Y)|X)\left[\dot\psi\left(\E(Y|X)\right)\right]^2\right) + \E\left(\psi(\E(\varphi(Y)|X))^2\right)
\end{eqnarray*}
where we have set $v(x)=\int \varphi(y)^2f(x,y)dy/\int f(x,y)dy$. This result and (\ref{hec}) give the desired form for $C(f)$ which ends the proof of Theorem \ref{tfec}.

\subsection{Proof of Theorem \ref{cramerrao2}}

We follow the proof of Theorem \ref{cramerrao1}. Assumptions A2 and A3 imply that
\begin{eqnarray*}
T(f)-T(f_0)&=&\iint \left(\big[\varphi(y)-m_0(x)\big]\dot\psi(m_0(x))+\psi(m_0(x))\right)\\
&&\Big(f(x,y)-f_0(x,y)\Big)dxdy+O\left(\int (f-f_0)^2\right)
\end{eqnarray*}
where $m_0(x)=\int \varphi(y)f_0(x,y)dy/\int f_0(x,y)dy$. This result shows that the Fr\'echet derivative of $T(f)$ at $f_0$ is $T'(f_0)\cdot h =\left< H(f_0,\cdot),h\right>$ where 
$$H(f_0,x,y)=\left(\big[\varphi(y)-m_0(x)\big]\dot\psi(m_0(x))+\psi(m_0(x))\right).$$
We then deduce that
\begin{eqnarray*}
K(h)&=& T'(f_0)\cdot \left(\sqrt{f_0}\left(h-\sqrt{f_0}\int h\sqrt{f}_0\right)\right)\\
&=& \int H(f_0,\cdot) \sqrt{f_0}h- \int H(f_0,\cdot) \sqrt{f_0} \int h\sqrt{f_0}\\
&=& \left<t,h\right>
\end{eqnarray*}
with
\begin{equation*}
t=H(f_0,\cdot)\sqrt{f_0}-\left(\int H(f_0,\cdot)f_0\right)\sqrt{f_0}.\\
\end{equation*}
The semiparametric Cram\'er-Rao bound for this problem is thus
\begin{equation*}
\|t\|_{L^2(dxdy)}^2=\int H(f_0,\cdot)^2 f_0 - \left(\int H(f_0,\cdot)f_0\right)^2= C(f_0)
\end{equation*}
where we recognize the expression of $C(f_0)$ in Theorem \ref{cramerrao2}.


\begin{thebibliography}{0}
\providecommand{\natexlab}[1]{#1}

\end{thebibliography}


\begin{thebibliography}{9}

\bibitem[Antoniadis(1984)]{anto84}
Antoniadis, A. (1984).
\newblock Analysis of variance on function spaces.
\newblock \emph{Math. Oper. Forsch. und Statist.}, series Statistics, 15(1):59--71.

\bibitem[Bayarri \etal(2007)]{bayber07}
Bayarii, M.J., Berger, J., Paulo, R., Sacks, J., Cafeo, J.A., Cavendish, J., Lin, C., and Tu, J. (2007).
\newblock A framework for validation of computer models.
\newblock {\em Technometrics}, 49:138--154.

\bibitem[Borgonovo(2007)]{borgo07}
Borgonovo E. (2007).
\newblock A New Uncertainty Importance Measure.
\newblock \emph{Reliability Engineering and System Saftey}, 92:771--784.

\bibitem[Carrasco \etal(2007)]{Carra07}
Carrasco, N., Banaszkiewicz, M., Thissen, R.,
  Dutuit, O., and Pernot, P. (2007).
\newblock Uncertainty analysis of bimolecular reactions in {T}itan ionosphere
  chemistry model.
\newblock \emph{Planetary and Space Science\/}, 55:141--157.

\bibitem[Chac\'on and Tenreiro(2011)]{chacon11}
Chac\'on, J.E. and Tenreiro C. (2011)
\newblock Exact and Asymptotically Optimal Bandwidths for Kernel Estimation of Density Functionals.
\newblock \emph{Methodol Comput Appl Probab}, DOI 10.1007/s11009-011-9243-x.

\bibitem[Cukier \etal(1973)]{CUK73}
Cukier, R.I., Fortuin, C.M., Shuler, K.E.,
  Petschek, A.G., and Schaibly, J.H. (1973).
\newblock Study of the sensitivity of coupled reaction systems to uncertainties
  in rate coefficients. {I} {T}heory.
\newblock \emph{The Journal of Chemical Physics\/}, 59:3873--3878.

\bibitem[Da Veiga \etal(2009)]{SDV09}
{Da Veiga}, S., Wahl, F., and Gamboa, F. (2006).
\newblock Local polynomial estimation for sensitivity analysis on models with
  correlated inputs.
\newblock \emph{Technometrics\/}, 59(4):452--463.

\bibitem[Fan and Gijbels(1996)]{FG96}
Fan, J. and Gijbels, I. (1996).
\newblock \emph{Local Polynomial Modelling and its Applications}. 
\newblock London: Chapman and Hall.

\bibitem[Ferrigno and Ducharme(2005)]{SF05}
Ferrigno, S. and Ducharme, G.R. (2005).
\newblock Un test d'ad\'equation global pour la fonction de r\'epartition
  conditionnelle.
\newblock \emph{Comptes rendus. Math\'ematique\/}, 341:313--316.

\bibitem[Gin\'e and Nickl(2008)]{gine2008a}
Gin\'e, E. and Nickl, R. (2008).
\newblock A simple adaptive estimator of the integrated square of a density.
\newblock \emph{Bernoulli}, 14(1):47--61

\bibitem[Gin\'e and Mason(2008)]{gine2008b}
Gin\'e, E. and Mason, D.M (2008).
\newblock Uniform in Bandwidth Estimation of Integral Functionals of the Density Function
\newblock \emph{Scandinavian Journal of Statistics,}, 35:739--761

\bibitem[Hoeffding(1948)]{hoeff48}
Hoeffding, W. (1948). 
\newblock A class of statistics with asymptotically normal distribution. 
\newblock \emph{The annals of Mathematical Statistics}, 19:293--32
5.
\bibitem[Houdret and Reynaud(2002)]{HR02}
{Houdr\'e}, C. and Reynaud, P. (2002).
\newblock Stochastic inequalities and applications.
\newblock In \emph{Euroconference on Stochastic inequalities and applications}.
  Birkhauser.

\bibitem[Ibragimov and Khas\'minskii(1991)]{IK91}
Ibragimov, I.A. and {Khas'minskii}, R.Z. (1991).
\newblock Asymptotically normal families of distributions and efficient
  estimation.
\newblock \emph{The Annals of Statistics\/},19:1681--1724.

\bibitem[Iooss \etal(2011)]{IMDR01}
Iooss, B., Marrel, A., Da Veiga, S. and Ribatet, M. (2011).
\newblock Global sensitivity analysis of stochastic computer models with joint metamodels
\newblock \emph{Stat Comput\/},DOI 10.1007/s11222-011-9274-8.

\bibitem[Iooss \etal(2006)]{IVD06}
Iooss, B., Van Dorpe, F. and Devictor, N. (2006).
\newblock Response surfaces and sensitivity analyses for an environmental model of dose calculations. 
\newblock \emph{Reliability Engineering and System Safety}, 91:1241-1251.

\bibitem[Janon \etal(2012)]{janon12}
Janon, A., Klein, T., {Lagnoux-Renaudie}, A., Nodet, M. and Prieur, C. (2012).
\newblock Asymptotic normality and efficiency of two Sobol index estimators.
\newblock HAL e-prints, {http://hal.inria.fr/hal-00665048}.

\bibitem[Kennedy and O'Hagan(2001)]{kenoha01}
Kennedy, M. and O'Hagan, A. (2001).
\newblock Bayesian calibration of computer models.
\newblock {\em Journal of the Royal Statistical Society}, 63(3):425--464.

\bibitem[Kerkyacharian and Picard(1996)]{KIKI96}
Kerkyacharian, G. and Picard, D. (1996).
\newblock Estimating nonquadratic functionals of a density using haar wavelets.
\newblock \emph{The Annals of Statistics\/}, 24:485--507.

\bibitem[Laurent(1996)]{BL96}
Laurent, B. (1996).
\newblock Efficient estimation of integral functionals of a density.
\newblock \emph{The Annals of Statistics\/}, 24:659--681.

\bibitem[Laurent(2005)]{BL05}
Laurent, B. (2005).
\newblock Adaptive estimation of a quadratic functional of a density by model
  selection.
\newblock \emph{ESAIM: Probability and Statistics\/}, 9:1--19.

\bibitem[Leonenko and Seleznjev(2010]{leo10}
Leonenko N. and Seleznjev O. (2010).
\newblock Statistical inference for the $\epsilon$-entropy and the quadratic Rényi entropy.
\newblock \emph{Journal of Multivariate Analysis}, 101:1981--1994.

\bibitem[Levit(1978)]{LEV78}
Levit, B.Y. (1978).
\newblock Asymptotically efficient estimation of nonlinear functionals.
\newblock \emph{Problems Inform. Transmission\/}, 14:204--209.

\bibitem[Li(1991)]{LI91}
Li, K.C. (1991).
\newblock Sliced inverse regression for dimension reduction.
\newblock \emph{Journal of the American Statistical Association\/}, 86:316--327.

\bibitem[Liu \etal(2006)]{LCS06}
Liu, H., Chen, W. and  Sudjianto, A. (2006).
\newblock Relative entropy based method for probabilistic sensitivity analysis in engineering design.
\newblock \emph{Journal of Mechanical Design}, 128(2):326--336.

\bibitem[Loubes \etal(2011)]{loubes11}
{Loubes}, J.-M. and {Marteau}, C. and {Solis}, M. and {Da Veiga}, S. (2011).
\newblock Efficient estimation of conditional covariance matrices for dimension reduction.
\newblock ArXiv e-prints, {http://adsabs.harvard.edu/abs/2011arXiv1110.3238L}.

\bibitem[Manceau \etal(2001)]{manmez01}
Manceau, E., Mezghani, M., Zabalza-Mezghani, I., and Roggero, F. (2001).
\newblock Combination of experimental design and joint modeling methods for quantifying the risk associated with deterministic and stochastic uncertainties - An integrated test study.
\newblock {\em 2001 SPE Annual Technical Conference and Exhibition, New Orleans, 30 September-3 October}, paper SPE 71620.

\bibitem[McKay(1995)]{MCK95}
McKay, M.D. (1995).
\newblock Evaluating prediction uncertainty.
\newblock Tech. Rep. NUREG/CR-6311, U.S. Nuclear Regulatory Commission and Los
  Alamos National Laboratory.

\bibitem[Oakley and O'Hagan(2004)]{OOH04}
Oakley, J.E. and O'Hagan, A. (2004).
\newblock Probabilistic sensitivity analysis of complex models : a bayesian
  approach.
\newblock \emph{Journal of the Royal Statistical Society Series B\/}, 66:751--769.

\bibitem[Owen(1994)]{owen94}
Owen, A.B. (1994).
\newblock Lattice sampling revisited: Monte Carlo variance of means over randomized orthogonal arrays.
\newblock \emph{The Annals of Statistics}, 22:930--945.

\bibitem[Saltelli \etal(2000)]{salcha00}
Saltelli, A., Chan, K., and Scott, E., editors (2000).
\newblock {\em Sensitivity analysis}.
\newblock Wiley Series in Probability and Statistics. Wiley.

\bibitem[Santner \etal.(2003)]{santner03}
Santner T., Williams B. and Notz W. (2003).
\newblock  The design and analysis of computer experiments.
\newblock  New York: Springer Verlag.

\bibitem[Sobol'(1993)]{SOB93}
Sobol', I~M. (1993).
\newblock Sensitivity estimates for nonlinear mathematical models.
\newblock \emph{MMCE\/}, 1:407--414.

\bibitem[Turanyi(1990)]{T90}
Turanyi, T. (1990).
\newblock Sensitivity analysis of complex kinetic systems.
\newblock \emph{Journal of Mathematical Chemistry\/}, 5:203--248.

\bibitem[Van Der Vaart(1998)]{VV98}
{Van Der Vaart}, A.W. (1998).
\newblock \emph{Asymptotic Statistics}.
\newblock Cambridge: Cambridge University Press.

\bibitem[Wand and Jones(1994)]{WJ94}
Wand, M. and Jones, M. (1994).
\newblock \emph{Kernel Smoothing}
\newblock London: Chapman and Hall.

\end{thebibliography}
\end{document}